\documentclass[12pt]{amsart}

\usepackage[utf8]{inputenc}
\usepackage[T1]{fontenc}
\usepackage[british]{babel}

\usepackage{textcomp}
\usepackage{scrhack}
\usepackage{xspace}
\usepackage{mparhack}
\usepackage{fixltx2e}

\usepackage{amsmath}
\usepackage[warning, all]{onlyamsmath}\AtBeginDocument{\catcode`\$=3}
\usepackage{amssymb}
\usepackage{stmaryrd}
\usepackage{dsfont}
\usepackage{mathtools}
\usepackage{mathdots}
\usepackage{gensymb}
\allowdisplaybreaks

\usepackage{graphicx}
\usepackage{subfig}
\usepackage{scalerel}
\usepackage{csquotes}
\usepackage{url}
\usepackage{tabularx}

\usepackage{caption}
\usepackage{tikz}
\usetikzlibrary{positioning, arrows, patterns}

\usepackage[capitalise]{cleveref} 

\newtheorem{theorem}{Theorem}[section]
\newtheorem{main-theorem}{Main Theorem}[section]
\newtheorem{corollary}[theorem]{Corollary}
\newtheorem{lemma}[theorem]{Lemma}

\theoremstyle{definition}
\newtheorem{definition}[theorem]{Definition}
\newtheorem{remark}[theorem]{Remark}
\newtheorem{example}[theorem]{Example}

\newtheorem*{proof-idea}{Proof Idea}

\makeatletter
\newcommand{\proofpart}[1]{%
  \par%
  \addvspace{\medskipamount}%
  \noindent\emph{#1.}
  \@afterheading%
}
\makeatother

\urldef{\mails}\path|simon.wacker@kit.edu|

\newcommand*{\define}[1]{\emph{#1}}

\providecommand\iff{\DOTSB\;\Longleftrightarrow\;}
\providecommand\implies{\DOTSB\;\Longrightarrow\;}

\DeclarePairedDelimiter\parens{\lparen}{\rparen}

\DeclarePairedDelimiter\abs{\lvert}{\rvert}

\DeclarePairedDelimiterX\inner[2]{\langle}{\rangle}{#1,#2} 
\DeclarePairedDelimiter\set{\{}{\}} 
\DeclarePairedDelimiter\net{\{}{\}}
\DeclarePairedDelimiter\family{\{}{\}}
\DeclarePairedDelimiter\sequence{\lparen}{\rparen}
\DeclarePairedDelimiter\equivclass{\lbrack}{\rbrack}

\DeclarePairedDelimiter\ceil{\lceil}{\rceil}

\mathchardef\breakingcomma\mathcode`\,
{\catcode`,=\active
  \gdef,{\breakingcomma\discretionary{}{}{}}
}
\newcommand*\ntuple[1]{\lparen\mathcode`\,=\string"8000 #1\rparen}

\newcommand*{\N}{\mathbb{N}}

\newcommand*{\R}{\mathbb{R}}

\newcommand*{\ball}{\mathbb{B}}
\newcommand*{\sphere}{\mathbb{S}}

\newcommand*{\euler}{\mathrm{e}} 

\newcommand*{\unityfnc}{\mathds{1}} 
\newcommand*{\nullityfnc}{0} 

\DeclareMathOperator{\identity}{id}

\DeclareMathOperator{\Exists}{\exists}
\DeclareMathOperator{\ForEach}{\forall}
\newcommand*\Holds{:}
\newcommand*\SuchThat{:}

\newcommand*\leftaction{\triangleright}

\newcommand*\rightsemiaction{\mathbin{\protect\scalerel*{\trianglelefteqslant}{\rhd}}} 

\DeclareMathOperator{\Sym}{Sym}

\makeatletter
\def\moverlay{\mathpalette\mov@rlay}
\def\mov@rlay#1#2{\leavevmode\vtop{%
   \baselineskip\z@skip \lineskiplimit-\maxdimen
   \ialign{\hfil$\m@th#1##$\hfil\cr#2\crcr}}}
\newcommand{\charfusion}[3][\mathord]{
    #1{\ifx#1\mathop\vphantom{#2}\fi
        \mathpalette\mov@rlay{#2\cr#3}
      }
    \ifx#1\mathop\expandafter\displaylimits\fi}
\makeatother
\newcommand{\cupdot}{\charfusion[\mathbin]{\cup}{\cdot}}

\newcommand*\boundary{\partial}

\DeclareMathOperator{\powerset}{\mathcal{P}}

\newcommand*\occurs{\sqsubseteq}
\newcommand*\semioccurs{\mathrel{\ooalign{$\occurs$\cr\hidewidth\raise.225ex\hbox{$\circ\mkern.5mu$}\cr}}} 

\newcommand*\suchthat{\mid} 
\newcommand*\from{\colon} 

\newcommand*\quotient{\slash}
\renewcommand*{\restriction}{\mathord{\upharpoonright}}

\newcommand*{\blank}{\mathord{\_}} 
\DeclareMathOperator{\vdegree}{deg}

\newcommand*\dominates{\succcurlyeq} 
\newcommand*\dominatedby{\preccurlyeq} 
\newcommand*\equivalent{\sim}
\newcommand*\inequivalent{\nsim}

\newcommand*{\graffito}[1]{}
\newcommand*{\mathnote}[1]{}
\renewcommand*{\index}[1]{}

\begin{document}
  \baselineskip=17pt

  \title[Right Amenability And Growth]{Right Amenability And Growth Of Finitely Right Generated Left Group Sets}
  \author{Simon Wacker}
  \address{%
    Simon Wacker\\
    Department of Informatics\\
    Karlsruhe Institute of Technology\\
    Am Fasanengarten 5\\
    76131 Karlsruhe\\
    Germany
  }
  \email{simon.wacker@kit.edu}


  \maketitle

  \begin{abstract}
    We introduce right generating sets, Cayley graphs, growth functions, types and rates, and isoperimetric constants for left homogeneous spaces equipped with coordinate systems; characterise right amenable finitely right generated left homogeneous spaces with finite stabilisers as those whose isoperimetric constant is $0$; and prove that finitely right generated left homogeneous spaces with finite stabilisers of sub-exponential growth are right amenable, in particular, quotient sets of groups of sub-exponential growth by finite subgroups are right amenable.
    %
  \end{abstract}

  The notion of amenability for groups was introduced by John von Neumann in 1929. It generalises the notion of finiteness. A group $G$ is \emph{left} or \emph{right amenable} if there is a finitely additive probability measure on $\powerset(G)$ that is invariant under left and right multiplication respectively. Groups are left amenable if and only if they are right amenable. A group is \emph{amenable} if it is left or right amenable.

  The definitions of left and right amenability generalise to left and right group sets respectively. A left group set $\ntuple{M, G, \leftaction}$ is \emph{left amenable} if there is a finitely additive probability measure on $\powerset(M)$ that is invariant under $\leftaction$. There is in general no natural action on the right that is to a left group action what right multiplication is to left group multiplication. Therefore, for a left group set there is no natural notion of right amenability.

  A transitive left group action $\leftaction$ of $G$ on $M$ induces, for each element $m_0 \in M$ and each family $\family{g_{m_0, m}}_{m \in M}$ of elements in $G$ such that, for each point $m \in M$, we have $g_{m_0, m} \leftaction m_0 = m$, a right quotient set semi-action $\rightsemiaction$ of $G \quotient G_0$ on $M$ with defect $G_0$ given by $m \rightsemiaction g G_0 = g_{m_0, m} g g_{m_0, m}^{-1} \leftaction m$, where $G_0$ is the stabiliser of $m_0$ under $\leftaction$. Each of these right semi-actions is to the left group action what right multiplication is to left group multiplication. They occur in the definition of global transition functions of cellular automata over left homogeneous spaces as defined in \cite{wacker:automata:2016}. A \emph{cell space} is a left group set together with choices of $m_0$ and $\family{g_{m_0, m}}_{m \in M}$.

  A cell space is \emph{right amenable} if there is a finitely additive probability measure on $\powerset(M)$ that is semi-invariant under $\rightsemiaction$. For example cell spaces with finite sets of cells, abelian groups, and finitely right generated cell spaces with finite stabilisers of sub-exponential growth are right amenable, in particular, quotients of finitely generated groups of sub-exponential growth by finite subgroups acted on by left multiplication. A net of non-empty and finite subsets of $M$ is a \emph{right Følner net} if, broadly speaking, these subsets are asymptotically invariant under $\rightsemiaction$. A finite subset $E$ of $G \quotient G_0$ and two partitions $\family{A_e}_{e \in E}$ and $\family{B_e}_{e \in E}$ of $M$ constitute a \emph{right paradoxical decomposition} if the map $\blank \rightsemiaction e$ is injective on $A_e$ and $B_e$, and the family $\family{(A_e \rightsemiaction e) \cupdot (B_e \rightsemiaction e)}_{e \in E}$ is a partition of $M$. The Tarski-Følner theorem states that right amenability, the existence of right Følner nets, and the non-existence of right paradoxical decompositions are equivalent. We prove it in \cite{wacker:amenable:2016} for cell spaces with finite stabilisers.

  A cell space $\mathcal{R}$ is \emph{finitely right generated} if there is a finite subset $S$ of $G \quotient G_0$ such that, for each point $m \in M$, there is a family $\family{s_i}_{i \in \set{1,2,\dotsc,k}}$ of elements in $S \cup S^{-1}$ such that $m = (((m_0 \rightsemiaction s_1) \rightsemiaction s_2) \rightsemiaction \dotsb) \rightsemiaction s_k$. The finite right generating set $S$ induces the \emph{$S$-Cayley graph} structure on $M$: For each point $m \in M$ and each generator $s \in S$, there is an edge from $m$ to $m \rightsemiaction s$. The length of the shortest path between two points of $M$ yields the \emph{$S$-metric}. The ball of radius $\rho \in \N_0$ centred at $m \in M$, denoted by $\ball_S(m, \rho)$, is the set of all points whose distance to $m$ is less than or equal to $\rho$. The \emph{$S$-growth function} is the map $\gamma_S \from \N_0 \to \N_0$, $k \mapsto \abs{\ball_S(m, k)}$; the \emph{growth type of $\mathcal{R}$}, which does not depend on $S$, is the equivalence class $\equivclass{\gamma_S}_\equivalent$, where two growth functions are equivalent if they dominate each other; and the \emph{$S$-growth rate} is the limit point of the sequence $\sequence{\sqrt[k]{\gamma_S(k)}}_{k \in K}$.

  A finitely right generated cell space $\mathcal{R}$ is said to have \emph{sub-exponential growth} if its growth type is not $\equivclass{\exp}_\equivalent$, which is the case if and only if its growth rates are $1$. The \emph{$S$-isoperimetric constant} is a real number between $0$ and $1$ that measures, broadly speaking, the invariance under $\rightsemiaction\restriction_{M \times S}$ that a finite subset of $M$ can have, where $0$ means maximally and $1$ minimally invariant. In the case that $G_0$ is finite, this constant is $0$ if and only if $\mathcal{R}$ is right amenable, and if $\mathcal{R}$ has sub-exponential growth, then it is right amenable, and if $G$ has sub-exponential growth, then so has $\mathcal{R}$.

  Cayley graphs were introduced by Arthur Cayley in his paper \enquote{Desiderata and suggestions: No. 2. The Theory of groups: graphical representation}\cite{cayley:1878}. The notion of growth was introduced by Vadim Arsenyevich Efremovich and Albert S. Švarc in their papers \enquote{The geometry of proximity}\cite{efremovich:1952} and \enquote{A volume invariant of coverings}\cite{svarc:1955}. Mikhail Leonidovich Gromov was the first to study groups through their word metrics, see for example his paper \enquote{Infinite Groups as Geometric Objects}\cite{gromov:1984}. The present paper is greatly inspired by the monograph \enquote{Cellular Automata and Groups}\cite{ceccherini-silberstein:coornaert:2010} by Tullio Ceccherini-Silberstein and Michel Coornaert. 

  In \cref{sec:gen-set} we introduce right generating sets. In \cref{sec:multigraphs} we recapitulate directed multigraphs. In \cref{sec:cayley} we introduce Cayley graphs induced by right generating sets. In \cref{sec:metrics} we introduce metrics and lengths induced by Cayley graphs. In \cref{sec:balls-and-spheres} we consider balls and spheres induced by metrics. In \cref{sec:interiors-closures-boundaries} we consider interiors, closures, and boundaries of any thickness of sets. In \cref{sec:growth-functions-and-types} we recapitulate growth functions and types. In \cref{sec:cell-spaces-growth-fnc-and-types} we introduce growth functions and types of cell spaces. In \cref{sec:growth-rates} we introduce growth rates of cell spaces. In \cref{sec:folner-cond} we prove that right amenability and having isoperimetric constant $0$ are equivalent, and we characterise right Følner nets. And in \cref{sec:subexp-growth-amenability} we prove that having sub-exponential growth implies right amenability.

  \subsubsection*{Preliminary Notions.} A \define{left group set} is a triple $\ntuple{M, G, \leftaction}$, where $M$ is a set, $G$ is a group, and $\leftaction$ is a map from $G \times M$ to $M$, called \define{left group action of $G$ on $M$}, such that $G \to \Sym(M)$, $g \mapsto [g \leftaction \blank]$, is a group homomorphism. The action $\leftaction$ is \define{transitive} if $M$ is non-empty and for each $m \in M$ the map $\blank \leftaction m$ is surjective; and \define{free} if for each $m \in M$ the map $\blank \leftaction m$ is injective. For each $m \in M$, the set $G \leftaction m$ is the \define{orbit of $m$}, the set $G_m = (\blank \leftaction m)^{-1}(m)$ is the \define{stabiliser of $m$}, and, for each $m' \in M$, the set $G_{m, m'} = (\blank \leftaction m)^{-1}(m')$ is the \define{transporter of $m$ to $m'$}.

  A \define{left homogeneous space} is a left group set $\mathcal{M} = \ntuple{M, G, \leftaction}$ such that $\leftaction$ is transitive. A \define{coordinate system for $\mathcal{M}$} is a tuple $\mathcal{K} = \ntuple{m_0, \family{g_{m_0, m}}_{m \in M}}$, where $m_0 \in M$ and for each $m \in M$ we have $g_{m_0, m} \leftaction m_0 = m$. The stabiliser $G_{m_0}$ is denoted by $G_0$. The tuple $\mathcal{R} = \ntuple{\mathcal{M}, \mathcal{K}}$ is a \define{cell space}. The map $\rightsemiaction \from M \times G \quotient G_0 \to M$, $(m, g G_0) \mapsto g_{m_0, m} g g_{m_0, m}^{-1} \leftaction m\ (= g_{m_0, m} g \leftaction m_0)$ is a \define{right semi-action of $G \quotient G_0$ on $M$ with defect $G_0$}, which means that
  \begin{gather*}
    \ForEach m \in M \Holds m \rightsemiaction G_0 = m,\\
    \ForEach m \in M \ForEach g \in G \Exists g_0 \in G_0 \SuchThat \ForEach \mathfrak{g}' \in G \quotient G_0 \Holds
          m \rightsemiaction g \cdot \mathfrak{g}' = (m \rightsemiaction g G_0) \rightsemiaction g_0 \cdot \mathfrak{g}'.
  \end{gather*}
  It is \define{transitive}, which means that the set $M$ is non-empty and for each $m \in M$ the map $m \rightsemiaction \blank$ is surjective; and \define{free}, which means that for each $m \in M$ the map $m \rightsemiaction \blank$ is injective; and \define{semi-commutes with $\leftaction$}, which means that
  \begin{equation*}
    \ForEach m \in M \ForEach g \in G \Exists g_0 \in G_0 \SuchThat \ForEach \mathfrak{g}' \in G \quotient G_0 \Holds
          (g \leftaction m) \rightsemiaction \mathfrak{g}' = g \leftaction (m \rightsemiaction g_0 \cdot \mathfrak{g}').
  \end{equation*}
  The maps $\iota \from M \to G \quotient G_0$, $m \mapsto G_{m_0, m}$, and $m_0 \rightsemiaction \blank$ are inverse to each other. Under the identification of $M$ with $G \quotient G_0$ by either of these maps, we have $\rightsemiaction \from (m, \mathfrak{g}) \mapsto g_{m_0, m} \leftaction \mathfrak{g}$.

  A left homogeneous space $\mathcal{M}$ is \define{right amenable} if there is a coordinate system $\mathcal{K}$ for $\mathcal{M}$ and there is a finitely additive probability measure $\mu$ on $M$ such that 
  \begin{equation*}
    \ForEach \mathfrak{g} \in G \quotient G_0 \ForEach A \subseteq M \Holds \parens[\big]{(\blank \rightsemiaction \mathfrak{g})\restriction_A \text{ injective} \implies \mu(A \rightsemiaction \mathfrak{g}) = \mu(A)},
  \end{equation*}
  in which case the cell space $\mathcal{R} = \ntuple{\mathcal{M}, \mathcal{K}}$ is called \define{right amenable}. When the stabiliser $G_0$ is finite, that is the case if and only if there is a \define{right Følner net in $\mathcal{R}$ indexed by $(I, \leq)$}, which is a net $\net{F_i}_{i \in I}$ in $\set{F \subseteq M \suchthat F \neq \emptyset, F \text{ finite}}$ such that
  \begin{equation*}
    \ForEach \mathfrak{g} \in G \quotient G_0 \Holds \lim_{i \in I} \frac{\abs{F_i \smallsetminus (\blank \rightsemiaction \mathfrak{g})^{-1}(F_i)}}{\abs{F_i}} = 0.
  \end{equation*}

  \section{Right Generating Sets}
  \label{sec:gen-set}


  In this section, let $\mathcal{R} = \ntuple{\ntuple{M, G, \leftaction}, \ntuple{m_0, \family{g_{m_0, m}}_{m \in M}}}$ be a cell space. 

  In \cref{def:right-generating-set} we define right generating sets of $\mathcal{R}$. And in \cref{lem:gen-set-of-group-induces-right-gen-set-of-action} we show how generating sets of $G$ induce right ones of $\mathcal{R}$.

  \begin{definition}
  \label{def:right-generating-set}
    Let $S$ be a subset of $G \quotient G_0$ such that $G_0 \cdot S \subseteq S$.
    \begin{enumerate}
      \item The set $\set{g^{-1} G_0 \suchthat s \in S, g \in s}$ is denoted by $S^{-1}$.
      \item The set $S$ is said to \define{right generate $\mathcal{R}$}\graffito{$S$ right generates $\mathcal{R}$}, called \define{right generating set of $\mathcal{R}$}\graffito{right generating set $S$ of $\mathcal{R}$}, and each element $s \in S$ is called \define{right generator}\graffito{right generator $s$} if and only if, for each element $m \in M$, there is a non-negative integer $k \in \N_0$ and there is a family $\family{s_i}_{i \in \{1,2,\dotsc,k\}}$ of elements in $S \cup S^{-1}$ such that
            \begin{equation*}
              \Big(\big((m_0 \rightsemiaction s_1) \rightsemiaction s_2\big) \rightsemiaction \dotsb\Big) \rightsemiaction s_k = m. 
            \end{equation*}
      \item The set $S$ is called \define{symmetric}\graffito{symmetric} if and only if $S^{-1} \subseteq S$.
    \end{enumerate}
  \end{definition}



  \begin{definition}
    The cell space $\mathcal{R}$ is called \define{finitely right generated}\graffito{finitely right generated cell space $\mathcal{R}$} if and only if there is a right generating set of $\mathcal{R}$ that is finite. 
  \end{definition}


  \begin{remark}
    If $S$ is a right generating set of $\mathcal{R}$, then $S \cup S^{-1}$ is a symmetric one; and, if $S$ is also finite and $G_0$ is finite, then $S \cup S^{-1}$ is finite. 
  \end{remark}

  \begin{lemma}
  \label{lem:gen-set-of-group-induces-right-gen-set-of-action}
    Let $T$ be a generating set of $G$. The set $S = \set{g_0 \cdot t G_0 \suchthat g_0 \in G_0, t \in T}$ is a right generating set of $\mathcal{R}$. And, if $T$ is symmetric, then so is $S$. And, if $T$ and $G_0$ are finite, then so is $S$. 
  \end{lemma}

  \begin{proof}
    Let $m \in M$. Then, because $\rightsemiaction$ is transitive, there is a $g \in G$ such that $m_0 \rightsemiaction g G_0 = m$. Moreover, there is a $k \in \N_0$ and there is a $\family{t_i}_{i \in \set{1,2,\dotsc,k}} \subseteq T \cup T^{-1}$ such that $t_1 t_2 \dotsb t_k = g$. Furthermore, there is a $\family{g_{i,0}}_{i \in \set{2,3,\dotsc,k}} \subseteq G_0$ such that 
    \begin{equation*}
      \Big(\big((m_0 \rightsemiaction t_1 G_0) \rightsemiaction g_{2,0} t_2 G_0\big) \rightsemiaction \dotsb\Big) \rightsemiaction g_{k,0} t_k G_0
      = m_0 \rightsemiaction t_1 t_2 \dotsb t_k G_0
      = m.
    \end{equation*}
    In conclusion, because $t_1 G_0 \in S \cup S^{-1}$ and $\family{g_{i,0} t_i G_0}_{i \in \set{2,3,\dotsc,k}} \subseteq S \cup S^{-1}$, the set $S$ is a right generating set of $\mathcal{R}$.

    Let $T$ be symmetric. Furthermore, let $s \in S$ and let $g \in s$. Then, there is a $g_0 \in G_0$, there is a $t \in T$, and there is a $g_0' \in G_0$ such that $g_0 \cdot t G_0 = s$ and $g_0 t g_0' = g$. Hence, because $(g_0')^{-1} \in G_0$ and $t^{-1} \in T$,
    \begin{equation*}
      g^{-1} G_0 = (g_0')^{-1} t^{-1} g_0^{-1} G_0 = (g_0')^{-1} \cdot t^{-1} G_0 \in S.
    \end{equation*}
    In conclusion, $S^{-1} \subseteq S$.

    If $T$ and $G_0$ are finite, then so is $S$. 
  \end{proof}

  \section{Directed Multigraphs} 
  \label{sec:multigraphs}

  \begin{definition} 
    Let $V$ and $E$ be two sets, and let $\sigma$ and $\tau$ be two maps from $E$ to $V$. The quadruple $\mathcal{G} = \ntuple{V, E, \sigma, \tau}$ is called \define{directed multigraph}\graffito{directed multigraph $\mathcal{G}$}; each element $v \in V$ is called \define{vertex}\graffito{vertex $v$}; each element $e \in E$ is called \define{edge from $\sigma(e)$ to $\tau(e)$}\graffito{edge $e$ from $\sigma(e)$ to $\tau(e)$}; for each element $e \in E$, the vertex $\sigma(e)$ is called \define{source of $e$}\graffito{source $\sigma(e)$ of $e$} and the vertex $\tau(e)$ is called \define{target of $e$}\graffito{target $\tau(e)$ of $e$}. 
  \end{definition}

  \begin{definition}
    Let $\mathcal{G} = \ntuple{V, E, \sigma, \tau}$ be a directed multigraph and let $e$ be an edge of $\mathcal{G}$. The edge $e$ is called \define{loop}\graffito{loop $e$} if and only if $\tau(e) = \sigma(e)$.
  \end{definition}

  \begin{definition}
    Let $\mathcal{G} = \ntuple{V, E, \sigma, \tau}$ be a directed multigraph and let $v$ be a vertex of $\mathcal{G}$.
    \begin{enumerate}
      \item The cardinal number
            \begin{equation*}
              \vdegree^+(v) = \abs{\set{e \in E \suchthat \sigma(e) = v}} \mathnote{out-degree $\vdegree^+(v)$ of $v$}
            \end{equation*}
            is called \define{out-degree of $v$}.
      \item The cardinal number
            \begin{equation*}
              \vdegree^-(v) = \abs{\set{e \in E \suchthat \tau(e) = v}} \mathnote{in-degree $\vdegree^-(v)$ of $v$}
            \end{equation*}
            is called \define{in-degree of $v$}.
      \item The cardinal number
            \begin{equation*}
              \vdegree(v) = \vdegree^+(v) + \vdegree^-(v)
            \end{equation*}
            is called \define{degree of $v$}\graffito{degree $\vdegree(v)$ of $v$}. \qedhere 
    \end{enumerate}
  \end{definition}

  \begin{definition}
    Let $\mathcal{G}$ be a directed multigraph, and let $v$ and $v'$ be two vertices of $\mathcal{G}$. The vertices $v$ and $v'$ are called \define{adjacent}\graffito{adjacent vertices $v$ and $v'$} if and only if there is an edge from $v$ to $v'$ or one from $v'$ to $v$.
  \end{definition}

  \begin{definition} 
    Let $\mathcal{G} = \ntuple{V, E, \sigma, \tau}$ be a directed multigraph and let $p = \sequence{e_i}_{i \in \set{1,2,\dotsc,k}}$ be a finite sequence of edges of $\mathcal{G}$. The sequence $p$ is called \define{path from $\sigma(e_1)$ to $\tau(e_k)$}\graffito{path $p$ from $\sigma(e_1)$ to $\tau(e_k)$} if and only if, for each index $i \in \set{1,2,\dotsc,k-1}$, we have $\tau(e_i) = \sigma(e_{i + 1})$.
  \end{definition}

  \begin{definition}
    Let $\mathcal{G}$ be a directed multigraph and let $p = \sequence{e_i}_{i \in \set{1,2,\dotsc,k}}$ be a path in $\mathcal{G}$. The number $\abs{p} = k$ is called \define{length of $p$}\graffito{length $\abs{p}$ of $p$}.
  \end{definition}

  \begin{definition}
    Let $\mathcal{G} = \ntuple{V, E, \sigma, \tau}$ be a directed multigraph. It is called
    \begin{enumerate}
      \item \define{symmetric}\graffito{symmetric} if and only if, for each edge $e \in E$, there is an edge $e' \in E$ such that $\sigma(e') = \tau(e)$ and $\tau(e') = \sigma(e)$; 
      \item \define{strongly connected}\graffito{strongly connected} if and only if, for each vertex $v \in V$ and each vertex $v' \in V$, there is a path $p$ from $v$ to $v'$; 
      \item \define{regular}\graffito{regular} if and only if all vertices of $\mathcal{G}$ have the same degree and, for each vertex $v \in V$, we have $\vdegree^-(v) = \vdegree^+(v)$. \qedhere 
    \end{enumerate}
  \end{definition}

  \begin{definition} 
    Let $\mathcal{G} = \ntuple{V, E, \sigma, \tau}$ be a directed multigraph, let $W$ be a subset of $V$, let $F$ be the set $\set{e \in E \suchthat \sigma(e), \tau(e) \in W}$, let $\varsigma$ be the map $\sigma\restriction_{F \to W}$, and let $\upsilon$ be the map $\tau\restriction_{F \to W}$. The subgraph $\mathcal{G}[W] = \ntuple{W, F, \varsigma, \upsilon}$ of $\mathcal{G}$ is called \define{induced by $W$}\graffito{subgraph $\mathcal{G}[W]$ of $\mathcal{G}$ induced by $W$}. 
  \end{definition}

  \begin{definition} 
    Let $\mathcal{G} = \ntuple{V, E, \sigma, \tau}$ be a symmetric and strongly connected directed multigraph. The map
    \begin{align*}
      d \from V \times V &\to     \N_0, \mathnote{distance $d$ on $\mathcal{G}$}\\
                 (v, v') &\mapsto \min\set{\abs{p} \suchthat \text{$p$ path from $v$ to $v'$}},
    \end{align*}
    is a metric on $V$ and called \define{distance on $\mathcal{G}$}.
  \end{definition}

  \begin{definition} 
    Let $\ntuple{V, E, \sigma, \tau}$ be a directed multigraph, let $\Lambda$ be a set, and let $\lambda$ be a map from $E$ to $\Lambda$. The quintuple $\mathcal{G} = \ntuple{V, E, \sigma, \tau, \lambda}$ is called \define{$\Lambda$-edge-labelled directed multigraph}\graffito{$\Lambda$-edge-labelled directed multigraph $\mathcal{G}$}.
  \end{definition}

  \section{Cayley Graphs}
  \label{sec:cayley}

  In this section, let $\mathcal{R} = \ntuple{\ntuple{M, G, \leftaction}, \ntuple{m_0, \family{g_{m_0, m}}_{m \in M}}}$ be a cell space and let $S$ be a right generating set of $\mathcal{R}$. 

  \begin{definition} 
    Let $E$ be the set $\set{(m, s, m \rightsemiaction s) \suchthat m \in M, s \in S}$, and let $\sigma \colon E \to M$, $\lambda \colon E \to S$, and $\tau \colon E \to M$ be the projections to the first, second, and third component respectively. The $S$-edge-labelled directed multigraph $\mathcal{G} = \ntuple{M, E, \sigma, \tau, \lambda}$ is called \define{$S$-Cayley graph of $\mathcal{R}$}\graffito{$S$-Cayley graph $\mathcal{G}$ of $\mathcal{R}$}.
  \end{definition} 

  \begin{remark}
    Let $\mathcal{G}$ be the $S$-Cayley graph of $\mathcal{R}$.
    \begin{enumerate}
      \item If $S$ is symmetric, then $\mathcal{G}$ is symmetric and strongly connected. 
      \item The following statements are equivalent:
            \begin{enumerate}
              \item $G_0 \in S$;
              \item At least one vertex of $\mathcal{G}$ has a loop;
              \item All vertices of $\mathcal{G}$ have a loop.
            \end{enumerate}
      \item Because $\rightsemiaction$ is free, there are no multiple edges in $\mathcal{G}$. \qedhere
    \end{enumerate}
  \end{remark}

  \begin{remark}
    Let $\mathcal{G}$ be the $S$-Cayley graph of $\mathcal{R}$, and let $m$ and $m'$ be two vertices of $\mathcal{G}$. The vertices $m$ and $m'$ are adjacent if and only if there is an element $s \in S$ such that $m \rightsemiaction s = m'$. 
  \end{remark}

  \begin{remark}
    Let $\mathcal{G}$ be the $S$-Cayley graph of $\mathcal{R}$ and let $m$ be a vertex of $\mathcal{G}$. The map
    \begin{align*}
      S &\to     m \rightsemiaction S,\\
      s &\mapsto m \rightsemiaction s,
    \end{align*}
    is a bijection onto the out-neighbourhood of $m$. It is injective, because $\rightsemiaction$ is free, and it is surjective, by definition. Therefore, if $S$ is symmetric, then the degree of $m$ is $2 \abs{S}$ in cardinal arithmetic and the graph $\mathcal{G}$ is regular. 
  \end{remark}

  \section{Metrics and Lengths}
  \label{sec:metrics}

  In this section, let $\mathcal{R} = \ntuple{\ntuple{M, G, \leftaction}, \ntuple{m_0, \family{g_{m_0, m}}_{m \in M}}}$ be a cell space and let $S$ be a symmetric right generating set of $\mathcal{R}$.

  In \cref{def:metric,def:length} we define the $S$-metric $d_S$ and the $S$-length $\abs{\blank}_S$ on $\mathcal{R}$ induced by the $S$-Cayley graph. And in \cref{lem:metric-and-liberation,lem:metric-invariant-under-left-action} we show how the $S$-metric relates to the left group action $\leftaction$ and the right quotient set semi-action $\rightsemiaction$.


  \begin{definition} 
  \label{def:metric}
    The distance on the $S$-Cayley graph of $\mathcal{R}$ is called \define{$S$-metric on $\mathcal{R}$} and denoted by $d_S$.
  \end{definition} 

  \begin{remark}
    The $S$-metric on $\mathcal{R}$ is
    \begin{align*} 
      d_S \from M \times M &\to     \N_0, \mathnote{$S$-metric $d_S$ on $\mathcal{R}$}\\ 
                   (m, m') &\mapsto \min\{k \in \N_0 \suchthat{} \begin{aligned}[t]
                                                                          &\Exists \family{s_i}_{i \in \{1,2,\dotsc,k\}} \subseteq S \SuchThat\\ 
                                                                          &(((m \rightsemiaction s_1) \rightsemiaction s_2) \rightsemiaction \dotsb) \rightsemiaction s_k = m'\}.
                                                                        \end{aligned}
    \end{align*}
  \end{remark}


  \begin{lemma}
  \label{lem:metric-and-liberation}
    Let $m$ and $m'$ be two elements of $M$, and let $s$ be an element of $S$. Then, $d_S(m, m' \rightsemiaction s) \leq d_S(m, m') + 1$.
  \end{lemma}

  \begin{proof}
    Let $k = d_S(m, m')$. Then, there is a $\family{s_i}_{i \in \{1,2,\dotsc,k\}} \subseteq S$ such that $(((m \rightsemiaction s_1) \rightsemiaction s_2) \rightsemiaction \dotsb) \rightsemiaction s_k = m'$. Hence, $((((m \rightsemiaction s_1) \rightsemiaction s_2) \rightsemiaction \dotsb) \rightsemiaction s_k) \rightsemiaction s = m' \rightsemiaction s$. Therefore, $d_S(m, m' \rightsemiaction s) \leq d_S(m, m') + 1$.
  \end{proof}

  \begin{lemma} 
  \label{lem:metric-invariant-under-left-action}
    Let $m$ and $m'$ be two elements of $M$, and let $g$ be an element of $G$. Then, $d_S(g \leftaction m, g \leftaction m') = d_S(m, m')$.
  \end{lemma}

  \begin{proof}
    Let $k = d_S(g \leftaction m, g \leftaction m')$. Then, there is a $\family{s_i}_{i \in \{1,2,\dotsc,k\}} \subseteq S$ such that $((((g \leftaction m) \rightsemiaction s_1) \rightsemiaction s_2) \rightsemiaction \dotsb) \rightsemiaction s_k = g \leftaction m'$. Moreover, because $\leftaction$ and $\rightsemiaction$ semi-commute, for each $i \in \set{1,2,\dotsc,k}$, there is a $g_{i,0} \in G_0$, such that
    \begin{align*}
      &  \bigg(\Big(\big((g \leftaction m) \rightsemiaction s_1\big) \rightsemiaction s_2\Big) \rightsemiaction \dotsb\bigg) \rightsemiaction s_k\\
      &= \bigg(\Big(\big(g \leftaction (m \rightsemiaction g_{1,0} \cdot s_1)\big) \rightsemiaction s_2\Big) \rightsemiaction \dotsb\bigg) \rightsemiaction s_k\\
      &= \dotsb\\
      &= g \leftaction \bigg(\Big(\big((m \rightsemiaction g_{1,0} \cdot s_1) \rightsemiaction g_{2,0} \cdot s_2\big) \rightsemiaction \dotsb\Big) \rightsemiaction g_{k,0} \cdot s_k\bigg).
    \end{align*}
    Hence, $(((m \rightsemiaction g_{1,0} \cdot s_1) \rightsemiaction g_{2,0} \cdot s_2) \rightsemiaction \dotsb) \rightsemiaction g_{k,0} \cdot s_k = m'$. Therefore, $d_S(m, m') \leq k = d_S(g \leftaction m, g \leftaction m')$.

    Taking $g \leftaction m$ for $m$, $g \leftaction m'$ for $m'$, and $g^{-1}$ for $g$ yields $d_S(g \leftaction m, g \leftaction m') \leq d_S(g^{-1} \leftaction (g \leftaction m), g^{-1} \leftaction (g \leftaction m')) = d_S(m, m')$. 
    In conclusion, $d_S(g \leftaction m, g \leftaction m') = d_S(m, m')$.
  \end{proof}

  \begin{lemma}
  \label{lem:truncated-minmal-path-yields-minimal-path}
    Let $m$ and $m'$ be two elements of $M$, let $\family{s_i}_{i \in \set{1, 2, \dotsc, d_S(m, m')}}$ be a family of elements in $S$ such that $m' = (((m \rightsemiaction s_1) \rightsemiaction s_2) \rightsemiaction \dotsb) \rightsemiaction s_{d_S(m, m')}$, let $i$ be an element of $\set{0, 1, 2, \dotsc, d_S(m, m')}$, and let $m_i = (((m \rightsemiaction s_1) \rightsemiaction s_2) \rightsemiaction \dotsb) \rightsemiaction s_i$ Then, $d_S(m, m_i) = i$.
  \end{lemma}

  \begin{proof}
    By definition of $m_i$, we have $d_S(m, m_i) \leq i$ and $d_S(m_i, m') \leq d_S(m, m') - i$. Therefore, because $d_S(m, m') \leq d_S(m, m_i) + d_S(m_i, m')$, we have $d_S(m, m_i) \geq d_S(m, m') - d_S(m_i, m') \geq d_S(m, m') - (d_S(m, m') - i) = i$. In conclusion, $d_S(m, m_i) = i$.
  \end{proof}

  \begin{definition}
  \label{def:length}
    The map
    \begin{align*}
      \abs{\blank}_S \from M &\to     \N_0, \mathnote{$S$-length $\abs{\blank}_S$ on $\mathcal{R}$}\\ 
                           m &\mapsto d_S(m_0, m),
    \end{align*}
    is called \define{$S$-length on $\mathcal{R}$}. 
  \end{definition}

  \begin{remark}
    For each element $m \in M$, we have $\abs{m}_S = 0$ if and only if $m = m_0$.
  \end{remark}

  \section{Balls and Spheres}
  \label{sec:balls-and-spheres}

  In this section, let $\mathcal{R} = \ntuple{\ntuple{M, G, \leftaction}, \ntuple{m_0, \family{g_{m_0, m}}_{m \in M}}}$ be a cell space and let $S$ be a symmetric right generating set of $\mathcal{R}$.

  In \cref{def:balls-and-spheres} we define balls $\ball_S$ and spheres $\sphere_S$ in the $S$-metric on $\mathcal{R}$. And in the lemmata and corollaries of this section we show how balls, spheres, the left group action $\leftaction$, the right quotient set semi-action $\rightsemiaction$, and the $S$-metric relate to each other.

  \begin{definition} 
  \label{def:balls-and-spheres}
    Let $m$ be an element of $M$ and let $\rho$ be a non-negative integer.
    \begin{enumerate}
      \item The set
            \begin{equation*}
              \ball_S(m, \rho) = \set{m' \in M \suchthat d_S(m, m') \leq \rho} \mathnote{ball $\ball_S(m, \rho)$ of radius $\rho$ centred at $m$}
            \end{equation*}
            is called \define{ball of radius $\rho$ centred at $m$}. The ball of radius $\rho$ centred at $m_0$ is denoted by $\ball_S(\rho)$\graffito{ball $\ball_S(\rho)$ of radius $\rho$ centred at $m_0$}. 
      \item The set
            \begin{equation*}
              \sphere_S(m, \rho) = \set{m' \in M \suchthat d_S(m, m') = \rho} \mathnote{sphere $\sphere_S(m, \rho)$ of radius $\rho$ centred at $m$}
            \end{equation*}
            is called \define{sphere of radius $\rho$ centred at $m$}. The sphere of radius $\rho$ centred at $m_0$ is denoted by $\sphere_S(\rho)$\graffito{sphere $\sphere_S(\rho)$ of radius $\rho$ centred at $m_0$}. \qedhere 
    \end{enumerate}
  \end{definition}

  \begin{remark}
  \label{rem:spheres-expressed-in-terms-of-balls}
    For each element $m \in M$, we have $\sphere_S(m, 0) = \ball_S(m, 0)$, and, for each positive integer $\rho \in \N_+$, we have $\sphere_S(m, \rho) = \ball_S(m, \rho) \smallsetminus \ball_S(m, \rho - 1)$.
  \end{remark}

  \begin{remark}
    For each non-negative integer $\rho \in \N_0$,
    \begin{equation*}
      \ball_S(\rho) = \set{m \in M \suchthat \abs{m}_S \leq \rho}
    \end{equation*}
    and
    \begin{equation*}
      \sphere_S(\rho) = \set{m \in M \suchthat \abs{m}_S = \rho}. \qedhere
    \end{equation*}
  \end{remark}

  \begin{definition} 
    Let $\sequence{A_k}_{k \in \N_0}$ be a sequence of sets.
    \begin{enumerate}
      \item The set
            \begin{equation*}
              \liminf_{k \to \infty} A_k = \bigcup_{k \in \N_0} \bigcap_{\substack{j \in \N_0 \\ j \geq k}} A_j \mathnote{limit inferior $\liminf_{k \to \infty} A_k$ of $\sequence{A_k}_{k \in \N_0}$}
            \end{equation*}
            is called \define{limit inferior of $\sequence{A_k}_{k \in \N_0}$}.
      \item The set
            \begin{equation*}
              \limsup_{k \to \infty} A_k = \bigcap_{k \in \N_0} \bigcup_{\substack{j \in \N_0 \\ j \geq k}} A_j \mathnote{limit superior $\limsup_{k \to \infty} A_k$ of $\sequence{A_k}_{k \in \N_0}$}
            \end{equation*}
            is called \define{limit superior of $\sequence{A_k}_{k \in \N_0}$}.
      \item Let $A$ be a set. The sequence $\sequence{A_k}_{k \in \N_0}$ is said to \define{converge to $A$}\graffito{converge to $A$}, the set $A$ is called \define{limit set of $\sequence{A_k}_{k \in \N_0}$}\graffito{limit set $A$ of $\sequence{A_k}_{k \in \N_0}$}, and $A$ is denoted by $\lim_{k \to \infty} A_k$\graffito{$\lim_{k \to \infty} A_k$} if and only if $\liminf_{k \to \infty} A_k = \limsup_{k \to \infty} A_k = A$.
      \item The sequence $\sequence{A_k}_{k \in \N_0}$ is called \define{convergent}\graffito{convergent} if and only if $\liminf_{k \to \infty} A_k = \limsup_{k \to \infty} A_k$. \qedhere 
    \end{enumerate}
  \end{definition}

  \begin{lemma} 
    Let $\sequence{A_k}_{k \in \N_0}$ be a non-decreasing or non-increasing sequence of sets. It converges to $\bigcup_{k \in \N_0} A_k$ or $\bigcap_{k \in \N_0} A_k$ respectively. \qed
  \end{lemma}


  \begin{remark}
  \label{rem:ball-of-radius-0-contains-one-element-and-sequence-of-balls-is-monotonic}
    For each element $m \in M$, we have $\ball_S(m, 0) = \set{m}$, and the sequence $\sequence{\ball_S(m, \rho)}_{\rho \in \N_0}$ is non-decreasing with respect to inclusion and converges to $M$, and hence, for each non-negative integer $\rho$, 
    \begin{equation*}
      \bigcup_{\substack{\rho' \in \N_0 \\ \rho' \geq \rho}} \ball_S(m, \rho') = M.
    \end{equation*}
  \end{remark}

  \begin{remark} 
  \label{rem:upper-bound-for-cardinality-of-balls}
    For each element $m \in M$ and each non-negative integer $\rho \in \N_0$, in cardinal arithmetic,
    \begin{equation*}
      \abs{\ball_S(m, \rho)} \leq (1 + \abs{S})^\rho,
    \end{equation*}
    because the map
    \begin{align*}
         (\set{G_0} \cup S)^\rho &\to     \ball_S(m, \rho),\\
      (s_1, s_2, \dotsc, s_\rho) &\mapsto (((m \rightsemiaction s_1) \rightsemiaction s_2) \rightsemiaction \dotsb) \rightsemiaction s_\rho,
    \end{align*}
    is surjective and $\abs{\set{G_0} \cup S}^\rho \leq (1 + \abs{S})^\rho$.
  \end{remark}

  \begin{lemma}
  \label{lem:ball-liberation-included-in-ball-one-larger}
    Let $m$ be an element of $M$, let $\rho$ be a non-negative integer, and let $s$ be an element of $S$. Then, $\ball_S(m, \rho) \rightsemiaction s \subseteq \ball_S(m, \rho + 1)$.
  \end{lemma} 

  \begin{proof}
    Let $m' \in \ball_S(m, \rho) \rightsemiaction s$. Then, there is an $m'' \in \ball_S(m, \rho)$ such that $m'' \rightsemiaction s = m'$. Hence, according to \cref{lem:metric-and-liberation}, we have $d_S(m, m') = d_S(m, m'' \rightsemiaction s) \leq d_S(m, m'') + 1 \leq \rho + 1$. Therefore, $m' \in \ball_S(m, \rho + 1)$. In conclusion, $\ball_S(m, \rho) \rightsemiaction s \subseteq \ball_S(m, \rho + 1)$.
  \end{proof}

  %

  \begin{lemma}
  \label{lem:left-action-and-balls}
    Let $m$ be an element of $M$, let $\rho$ be a non-negative integer, and let $g$ be an element of $G$. Then, $g \leftaction \ball_S(m, \rho) = \ball_S(g \leftaction m, \rho)$.
  \end{lemma}

  \begin{proof} 
      First, let $m' \in g \leftaction \ball_S(m, \rho)$. Then, $g^{-1} \leftaction m' \in \ball_S(m, \rho)$ and thus $d_S(m, g^{-1} \leftaction m') \leq \rho$. Hence, according to \cref{lem:metric-invariant-under-left-action},
      \begin{align*}
        d_S(g \leftaction m, m') &= d_S(g^{-1} \leftaction (g \leftaction m), g^{-1} \leftaction m')\\
                                 &= d_S(m, g^{-1} \leftaction m')\\
                                 &\leq \rho.
      \end{align*}
      Therefore, $m' \in \ball_S(g \leftaction m, \rho)$. In conclusion, $g \leftaction \ball_S(m, \rho) \subseteq \ball_S(g \leftaction m, \rho)$.

      Secondly, let $m' \in \ball_S(g \leftaction m, \rho)$. Then, $d_S(g \leftaction m, m') \leq \rho$. Thus, according to \cref{lem:metric-invariant-under-left-action},
      \begin{align*}
        d_S(m, g^{-1} \leftaction m') &= d_S(g \leftaction m, g \leftaction (g^{-1} \leftaction m'))\\
                                      &= d_S(g \leftaction m, m')\\
                                      &\leq \rho.
      \end{align*}
      Hence, $g^{-1} \leftaction m' \in \ball_S(m, \rho)$. Therefore, $m' \in g \leftaction \ball_S(m, \rho)$. In conclusion, $\ball_S(g \leftaction m, \rho) \subseteq g \leftaction \ball_S(m, \rho)$.
  \end{proof}

  \begin{corollary}
    Let $m$ be an element of $M$, let $\rho$ be a non-negative integer, and let $g_m$ be an element of $G_m$. Then, $g_m \leftaction \ball_S(m, \rho) = \ball_S(m, \rho)$. In particular, $G_m \leftaction \ball_S(m, \rho) = \ball_S(m, \rho)$.
  \end{corollary}

  \begin{proof}
    Because $g_m \leftaction m = m$, this is a direct consequence of \cref{lem:left-action-and-balls}.
  %
  \end{proof}

  \begin{corollary} 
  \label{cor:balls-of-equal-radius-have-same-number-of-elements}
    Let $m$ and $m'$ be two elements of $M$, and let $\rho$ be a non-negative integer. Then, $\abs{\ball_S(m, \rho)} = \abs{\ball_S(m', \rho)}$.
  \end{corollary}

  \begin{proof}
    Because there is a $g \in G$ such that $g \leftaction m = m'$, and $g \leftaction \blank$ is injective, this is a direct consequence of \cref{lem:left-action-and-balls}.
  \end{proof}

  \begin{lemma} 
  \label{lem:liberation-under-identification-of-quotient-set-with-M} 
    Let $m$ and $m'$ be two elements of $M$ and identify $M$ with $G \quotient G_0$ by $[m \mapsto G_{m_0,m}]$. Then, $m \rightsemiaction m' = g_{m_0, m} \leftaction m'$.
  \end{lemma}

  \begin{proof}
    Let $g \in G_{m_0, m'}$. Then, $G_{m_0, m'} = g G_0$. Hence,
    \begin{align*}
      m \rightsemiaction m' &= m \rightsemiaction G_{m_0, m'}\\
                       &= m \rightsemiaction g G_0\\
                       &= g_{m_0, m} g \leftaction m_0\\
                       &= g_{m_0, m} \leftaction (g \leftaction m_0)\\
                       &= g_{m_0, m} \leftaction m'. \qedhere
    \end{align*}
  \end{proof}

  \begin{lemma} 
  \label{lem:almost-associativity-of-liberation-under-identification}
    Let $m$, $m'$, and $m''$ be three elements of $M$ and identify $M$ with $G \quotient G_0$ by $[m \mapsto G_{m_0, m}]$. Then, there is an element $g_0 \in G_0$ such that $(m \rightsemiaction m') \rightsemiaction m'' = m \rightsemiaction (m' \rightsemiaction (g_0 \leftaction m''))$. 
  \end{lemma}

  \begin{proof} 
    Because $\rightsemiaction$ is a right semi-action, there is an element $g_0 \in G_0$ such that $m \rightsemiaction g_{m_0, m'} \cdot g_0 \cdot G_{m_0, m''} = (m \rightsemiaction g_{m_0, m'} G_0) \rightsemiaction G_{m_0, m''}$. And, under the identification of $M$ with $G \quotient G_0$, we have $G_{m_0, m''} = m''$, $g_{m_0, m'} G_0 = G_{m_0, m'} = m'$, and $g_{m_0, m'} \cdot g_0 \cdot G_{m_0, m''} = m' \rightsemiaction (g_0 \leftaction m'')$. Therefore, $m \rightsemiaction (m' \rightsemiaction (g_0 \leftaction m'')) = (m \rightsemiaction m') \rightsemiaction m''$.
  \end{proof}

  \begin{corollary}
  \label{cor:ball-centred-at-mzero-to-m}
    Let $m$ be an element of $M$, let $\rho$ be a non-negative integer, and identify $M$ with $G \quotient G_0$ by $[m \mapsto G_{m_0,m}]$. Then, $m \rightsemiaction \ball_S(\rho) = \ball_S(m, \rho)$.
  \end{corollary}

  \begin{proof}
    According to \cref{lem:liberation-under-identification-of-quotient-set-with-M} and \cref{lem:left-action-and-balls},
    \begin{align*}
      m \rightsemiaction \ball_S(\rho) &= g_{m_0, m} \leftaction \ball_S(\rho)\\
                               &= g_{m_0, m} \leftaction \ball_S(m_0, \rho)\\
                               &= \ball_S(g_{m_0, m} \leftaction m_0, \rho)\\
                               &= \ball_S(m, \rho). \qedhere
    \end{align*}
  %
  \end{proof}

  \begin{corollary}
  \label{cor:liberation-of-balls-yields-bigger-one}
    Let $m$ be an element of $M$, let $\rho$ and $\rho'$ be two non-negative integers, and identify $M$ with $G \quotient G_0$ by $[m \mapsto G_{m_0,m}]$. Then, $\ball_S(m, \rho) \rightsemiaction \ball_S(\rho') = \ball_S(m, \rho + \rho')$.
  \end{corollary}

  \begin{proof}
      First, let $m' \in \ball_S(m, \rho) \rightsemiaction \ball_S(\rho')$. Then, there is an $m'' \in \ball_S(m, \rho)$ such that $m' \in m'' \rightsemiaction \ball_S(\rho')$. Moreover, according to \cref{cor:ball-centred-at-mzero-to-m}, we have $m'' \rightsemiaction \ball_S(\rho') = \ball_S(m'', \rho')$. Hence, because $d_S$ is subadditive, $d_S(m, m') \leq d_S(m, m'') + d_S(m'', m') \leq \rho + \rho'$. Therefore, $m' \in \ball_S(m, \rho + \rho')$. In conclusion, $\ball_S(m, \rho) \rightsemiaction \ball_S(\rho') \subseteq \ball_S(m, \rho + \rho')$.

      Secondly, let $m' \in \ball_S(m, \rho + \rho')$.
      \begin{description}
        \item[Case 1] $m' \in \ball_S(m, \rho)$. Then, because $m_0 \in \ball_S(\rho')$, we have $m' = m' \rightsemiaction m_0 \in \ball_S(m, \rho) \rightsemiaction \ball_S(\rho')$.
        \item[Case 2] $m' \notin \ball_S(m, \rho)$. Then, there is a $j \in \set{\rho + 1, \rho + 2, \dotsc, \rho + \rho'}$ and there is a $\family{s_i}_{i \in \set{1,2,\dotsc,j}} \subseteq S$ such that $(((m'' \rightsemiaction s_{\rho + 1}) \rightsemiaction s_{\rho + 2}) \rightsemiaction \dotsb) \rightsemiaction s_{\rho + \rho'} = m'$, where $m'' = (((m \rightsemiaction s_1) \rightsemiaction s_2) \rightsemiaction \dotsb) \rightsemiaction s_\rho \in \ball_S(m, \rho)$. Hence, $m' \in \ball_S(m'', \rho') = m'' \rightsemiaction \ball_S(\rho') \subseteq \ball_S(m, \rho) \rightsemiaction \ball_S(\rho')$.
      \end{description}
      In either case, $m' \in \ball_S(m, \rho) \rightsemiaction \ball_S(\rho')$. In conclusion, $\ball_S(m, \rho + \rho') \subseteq \ball_S(m, \rho) \rightsemiaction \ball_S(\rho')$.
  \end{proof}

  \begin{definition} 
    Let $A$ and $A'$ be two subsets of $M$. The non-negative number or infinity
    \begin{equation*}
      d_S(A, A') = \min\set{d_S(a, a') \suchthat a \in A, a' \in A'} \mathnote{distance $d_S(A, A')$ of $A$ and $A'$}
    \end{equation*}
    is called \define{distance of $A$ and $A'$}\graffito{distance of $A$ and $A'$}, where we put $\min\emptyset = \infty$. In the case that $A = \set{a}$, we write $d_S(a, A')$ in place of $d_S(\set{a}, A')$; and in the case that $A' = \set{a'}$, we write $d_S(A, a')$ in place of $d_S(A, \set{a'})$.
  \end{definition}

  \begin{lemma}
  \label{lem:distance-of-sphere-and-point}
    Let $m$ and $m'$ be two elements of $M$, and let $\rho$ be a non-negative integer such that $\rho \leq d_S(m, m')$. Then, $d_S(\sphere_S(m, \rho), m') = d_S(m, m') - \rho$.
  \end{lemma}

  \begin{proof} 
    Let $\rho' = d_S(m, m')$.

    Then, there is a $\family{s_i}_{i \in \set{1,2,\dotsc,\rho'}}$ such that $(((m \rightsemiaction s_1) \rightsemiaction s_2) \rightsemiaction \dotsb) \rightsemiaction s_{\rho'} = m'$. Let $m'' = (((m \rightsemiaction s_1) \rightsemiaction s_2) \rightsemiaction \dotsb) \rightsemiaction s_{\rho}$. Then, $(((m'' \rightsemiaction s_{\rho + 1}) \rightsemiaction s_{\rho + 2}) \rightsemiaction \dotsb) \rightsemiaction s_{\rho'} = m'$. And, according to \cref{lem:truncated-minmal-path-yields-minimal-path}, we have $m'' \in \sphere_S(m, \rho)$. Thus, $d_S(\sphere_S(m, \rho), m') \leq d_S(m'', m') \leq \rho' - \rho$.

    Suppose that $d_S(\sphere_S(m, \rho), m') < \rho' - \rho$. Then, there is an $m'' \in \sphere_S(m, \rho)$ such that $d_S(m'', m') < \rho' - \rho$. Hence, $d_S(m, m') \leq d_S(m, m'') + d_S(m'', m') < \rho + (\rho' - \rho) = \rho'$, which contradicts $d_S(m, m') = \rho'$. Therefore, $d_S(\sphere_S(m, \rho), m') \geq \rho' - \rho$.

    In conclusion, $d_S(\sphere_S(m, \rho), m') = \rho' - \rho = d_S(m, m') - \rho$.
  \end{proof}

  \begin{corollary}
  \label{cor:distance-of-spheres}
    Let $m$ be an element of $M$, and let $\rho$ and $\rho'$ be two non-negative integers such that the spheres $\sphere_S(m, \rho)$ and $\sphere_S(m, \rho')$ are non-empty. Then, $d_S(\sphere_S(m, \rho), \sphere_S(m, \rho')) = \abs{\rho - \rho'}$.
  \end{corollary}

  \begin{proof}
    Without loss of generality, let $\rho \leq \rho'$. Then, for each $m' \in \sphere_S(m, \rho')$, according to \cref{lem:distance-of-sphere-and-point}, we have $d_S(\sphere_S(m, \rho), m') = \rho' - \rho$. In conclusion, $d_S(\sphere_S(m, \rho), \sphere_S(m, \rho')) = \rho' - \rho = \abs{\rho - \rho'}$.
  \end{proof}

  \begin{corollary} 
    Let $m$ and $m'$ be two elements of $M$, and let $\rho$ be a non-negative integer. Then, $d_S(\sphere_S(m, \rho), m') \geq \abs{d_S(m, m') - \rho}$.
  \end{corollary}

  \begin{proof} 
    If $\sphere_S(m, \rho) = \emptyset$, then $d_S(\sphere_S(m, \rho), m') = \infty \geq \abs{d_S(m, m') - \rho}$. Otherwise, let $\rho' = d_S(m, m')$. Then, $\sphere_S(m, \rho') \neq \emptyset$. Hence, according to \cref{cor:distance-of-spheres}, we have $d_S(\sphere_S(m, \rho), m') \geq d_S(\sphere_S(m, \rho), \sphere_S(m, \rho')) = \abs{\rho - \rho'} = \abs{d_S(m, m') - \rho}$.
  \end{proof}

  \begin{lemma}
  \label{lem:distance-of-balls}
    Let $m$ and $m'$ be two elements of $M$, and let $\rho$ and $\rho'$ be two non-negative integers such that $\rho + \rho' \leq d_S(m, m')$. Then, $d_S(\ball_S(m, \rho), \ball_S(m', \rho')) = d_S(m, m') - (\rho + \rho')$.
  \end{lemma}

  \begin{proof}
    For each $m_\rho \in \ball_S(m, \rho)$ and each $m'_{\rho'} \in \ball_S(m', \rho')$, because $d_S$ is subadditive,
    \begin{align*}
      d_S(m, m') &\leq d_S(m, m_\rho) + d_S(m_\rho, m'_{\rho'}) + d_S(m'_{\rho'}, m')\\
                 &\leq \rho + d_S(m_\rho, m'_{\rho'}) + \rho',
    \end{align*}
    and hence $d_S(m_\rho, m'_{\rho'}) \geq d_S(m, m') - (\rho + \rho')$. Therefore, $d_S(\ball_S(m, \rho), \ball_S(m', \rho')) \geq d_S(m, m') - (\rho + \rho')$.

    Moreover, there is a $\family{s_i}_{i \in \set{1,2,\dotsc,d_S(m,m')}}$ such that $(((m \rightsemiaction s_1) \rightsemiaction s_2) \dotsb) \rightsemiaction s_{d_S(m,m')} = m'$. Let $m_\rho = ((m \rightsemiaction s_1) \rightsemiaction \dotsb) \rightsemiaction s_\rho$ and let $m'_{\rho'} = ((m_\rho \rightsemiaction s_{\rho + 1}) \rightsemiaction \dotsb) \rightsemiaction s_{d_S(m,m') - \rho'}$. Then, $m_\rho \in \ball_S(m, \rho)$ and, because $(((m'_{\rho'} \rightsemiaction s_{d_S(m,m') - \rho' + 1}) \rightsemiaction s_{d_S(m,m') - \rho' + 2}) \dotsb) \rightsemiaction s_{d_S(m,m')} = m'$, we have $m'_{\rho'} \in \ball_S(m', \rho')$. And, $d_S(m_\rho, m'_{\rho'}) \leq d_S(m,m') - \rho' - \rho$. Therefore, $d_S(\ball_S(m, \rho), \ball_S(m', \rho')) \leq d_S(m, m') - (\rho + \rho')$. 

    In conclusion, $d_S(\ball_S(m, \rho), \ball_S(m', \rho')) = d_S(m, m') - \rho - \rho'$.
  \end{proof}

  \section{Interiors, Closures, and Boundaries}
  \label{sec:interiors-closures-boundaries}

  In this section, let $\mathcal{R} = \ntuple{\ntuple{M, G, \leftaction}, \ntuple{m_0, \family{g_{m_0, m}}_{m \in M}}}$ be a cell space and let $S$ be a symmetric right generating set of $\mathcal{R}$.

  In \cref{def:k-interior-closure-and-boundary} we define $\theta$-interiors $A^{-\theta}$, $\theta$-closures $A^{+\theta}$, and (internal/external) $\theta$-boundaries $\boundary_\theta A$, $\boundary_\theta^- A$, or $\boundary_\theta^+ A$. And in the lemmata and corollaries of this section we characterise them and show how they and the $S$-metric relate to each other.

  \begin{definition}
  \label{def:k-interior-closure-and-boundary}
    Let $A$ be a subset of $M$, let $\theta$ be a non-negative integer, and identify $M$ with $G \quotient G_0$ by $[m \mapsto G_{m_0,m}]$.
    \begin{enumerate}
      \item The set
            \begin{equation*}
              A^{-\theta} = A^{-\ball_S(\theta)} \quad \big(= \set{m \in M \suchthat m \rightsemiaction \ball_S(\theta) \subseteq A}\big) 
            \end{equation*}
            is called \define{$\theta$-interior of $A$}\index{interior of $A$@$\theta$-interior of $A$}.
      \item The set
            \begin{equation*}
              A^{+\theta} = A^{+\ball_S(\theta)} \quad \big(= \set{m \in M \suchthat (m \rightsemiaction \ball_S(\theta)) \cap A \neq \emptyset}\big) 
            \end{equation*}
            is called \define{$\theta$-closure of $A$}\index{closure of $A$@$\theta$-closure of $A$}.
      \item The set
            \begin{equation*}
              \boundary_\theta A = A^{+\theta} \smallsetminus A^{-\theta} \quad \big(= A^{+\ball_S(\theta)} \smallsetminus A^{-\ball_S(\theta)} = \boundary_{\ball_S(\theta)} A\big) 
            \end{equation*}
            is called \define{$\theta$-boundary of $A$}\index{boundary of $A$@$\theta$-boundary of $A$}.
      \item The set
            \begin{equation*}
              \boundary_\theta^- A = A \smallsetminus A^{-\theta} 
            \end{equation*}
            is called \define{internal $\theta$-boundary of $A$}\index{$\theta$-boundary of $A$!internal}\index{boundary of $A$!internal}.
      \item The set
            \begin{equation*}
              \boundary_\theta^+ A = A^{+\theta} \smallsetminus A 
            \end{equation*}
            is called \define{external $\theta$-boundary of $A$}\index{$\theta$-boundary of $A$!external}\index{boundary of $A$!external}. \qedhere
    \end{enumerate}
  \end{definition}


  \begin{lemma}
  \label{lem:characterisation-of-k-closure-and-interior}
    Let $A$ be a subset of $M$ and identify $M$ with $G \quotient G_0$ by $[m \mapsto G_{m_0,m}]$. For each non-negative integer $\theta \in \N_0$,
    \begin{enumerate}
      \item\label{it:characterisation-of-k-closure-and-interior:interior}
            $\displaystyle A^{-\theta} = \set{m \in A \suchthat \ball_S(m, \theta) \subseteq A}$;
      \item\label{it:characterisation-of-k-closure-and-interior:closure}
            $\displaystyle A^{+\theta} = \bigcup_{m \in A} \ball_S(m, \theta) = A \rightsemiaction \ball_S(\theta)$. \qedhere
    \end{enumerate}
  \end{lemma}

  \begin{proof}
    Let $\theta \in \N_0$ and let $m \in M$.
    \begin{enumerate}
      \item 
            According to \cref{cor:ball-centred-at-mzero-to-m},
            \begin{equation*}
              A^{-\theta} = \set{m \in M \suchthat \ball_S(m, \theta) \subseteq A}.
            \end{equation*}
            Therefore, because $m \in \ball_S(m, \theta)$,
            \begin{equation*}
              A^{-\theta} = \set{m \in A \suchthat \ball_S(m, \theta) \subseteq A}.
            \end{equation*}
      \item According to \cref{cor:ball-centred-at-mzero-to-m},
            \begin{align*}
              A^{+\theta}
              &= \set{m \in M \suchthat (m \rightsemiaction \ball_S(\theta)) \cap A \neq \emptyset}\\
              &= \set{m \in M \suchthat \Exists m' \in A \SuchThat m' \in m \rightsemiaction \ball_S(\theta)}\\
              &= \set{m \in M \suchthat \Exists m' \in A \SuchThat m' \in \ball_S(m, \theta)}.
            \end{align*}
            Moreover, because of the symmetry of $d_S$, for each $m' \in A$, 
            \begin{align*}
              m' \in \ball_S(m, \theta) &\iff d_S(m, m') \leq \theta\\
                                   &\iff m \in \ball_S(m', \theta).
            \end{align*}
            Hence, according to \cref{cor:ball-centred-at-mzero-to-m},
            \begin{align*}
              A^{+\theta}
              &= \set{m \in M \suchthat \Exists m' \in A \SuchThat m \in \ball_S(m', \theta)}\\
              &= \bigcup_{m' \in A} \ball_S(m', \theta)\\
              &= \bigcup_{m' \in A} m' \rightsemiaction \ball_S(\theta)\\
              &= A \rightsemiaction \ball_S(\theta). \qedhere
            \end{align*}
    \end{enumerate}
  \end{proof}

  \begin{corollary} 
  \label{cor:characterisation-of-k-closure-and-interior-of-balls}
    Let $m$ be an element of $M$, let $\rho$ be a non-negative integer, and let $\theta$ be a non-negative integer. Then,
    \begin{enumerate}
      \item\label{it:characterisation-of-k-closure-and-interior-of-balls:interior}
            $\displaystyle \ball_S(m, \rho)^{-\theta} \supseteq \ball_S(m, \rho - \theta)$;
      \item\label{it:characterisation-of-k-closure-and-interior-of-balls:closure}
            $\displaystyle \ball_S(m, \rho)^{+\theta} = \ball_S(m, \rho + \theta)$;
      \item\label{it:characterisation-of-k-closure-and-interior-of-balls:boundary}
            $\displaystyle \boundary_\theta \ball_S(m, \rho) \subseteq \ball_S(m, \rho + \theta) \smallsetminus \ball_S(m, \rho - \theta)$.
    \end{enumerate}
  \end{corollary}

  \begin{proof}
    \begin{enumerate}
      \item According to \cref{cor:liberation-of-balls-yields-bigger-one}, we have $\ball_S(m, \rho - \theta) \rightsemiaction \ball_S(\theta) \subseteq \ball_S(m, \rho)$. Hence, according to \cref{def:k-interior-closure-and-boundary}, we have $\ball_S(m, \rho - \theta) \subseteq \ball_S(m, \rho)^{-\theta}$. 
      \item According to \cref{it:characterisation-of-k-closure-and-interior:closure} of \cref{lem:characterisation-of-k-closure-and-interior} and \cref{cor:liberation-of-balls-yields-bigger-one}, we have $\ball_S(m, \rho)^{+\theta} = \ball_S(m, \rho) \rightsemiaction \ball_S(\theta) = \ball_S(m, \rho + \theta)$.
      \item This is a direct consequence of \cref{it:characterisation-of-k-closure-and-interior-of-balls:interior,it:characterisation-of-k-closure-and-interior-of-balls:closure}. \qedhere
    \end{enumerate}
  \end{proof}

  \begin{lemma} 
  \label{lem:repeated-k-bounderies-etc} 
    Let $A$ be a subset of $M$, and let $\theta$ and $\theta'$ be two non-negative integers. The following statements hold:
    \begin{enumerate}
      \item \label{it:repeated-k-bounderies-etc:interior}
            $\displaystyle (A^{-\theta})^{-\theta'} = A^{-(\theta + \theta')}$;
      \item $\displaystyle \boundary_{\theta'}^- A^{-\theta} = A^{-\theta} \smallsetminus A^{-(\theta + \theta')}$;
      \item \label{it:repeated-k-bounderies-etc:closure}
            $\displaystyle (A^{+\theta})^{+\theta'} = A^{+(\theta + \theta')}$;
      \item \label{it:repeated-k-bounderies-etc:external-boundary}
            $\displaystyle \boundary_{\theta'}^+ A^{+\theta} = A^{+(\theta + \theta')} \smallsetminus A^{+\theta}$;
      \item \label{it:repeated-k-bounderies-etc:closure-interior}
            Let $\theta' \leq \theta$. Then, $A^{+(\theta - \theta')} \subseteq (A^{+\theta})^{-\theta'}$ and $(A^{-\theta})^{+\theta'} \subseteq A^{-(\theta - \theta')}$. \qedhere
    \end{enumerate}
  \end{lemma}

  \begin{proof} 
    \begin{enumerate}
      \item For each $m' \in A$, according to \cref{cor:ball-centred-at-mzero-to-m} and \cref{lem:characterisation-of-k-closure-and-interior}, we have $m' \in A^{-\theta}$ if and only if $m' \rightsemiaction \ball_S(\theta) = \ball_S(m', \theta) \subseteq A$. Hence, according to \cref{cor:liberation-of-balls-yields-bigger-one},
            \begin{align*}
              (A^{-\theta})^{-\theta'}
              &= \set{m' \in A \suchthat \ball_S(m', \theta') \subseteq A^{-\theta}}\\
              &= \set{m' \in A \suchthat \ball_S(m', \theta') \rightsemiaction \ball_S(\theta) \subseteq A}\\
              &= \set{m' \in A \suchthat \ball_S(m', \theta + \theta') \subseteq A}\\
              &= A^{-(\theta + \theta')}.
            \end{align*}
      \item According to \cref{it:repeated-k-bounderies-etc:interior},
            \begin{align*}
              \boundary_{\theta'}^- A^{-\theta}
              &= A^{-\theta} \smallsetminus (A^{-\theta})^{-\theta'}\\
              &= A^{-\theta} \smallsetminus A^{-(\theta + \theta')}.
            \end{align*}
      \item According to \cref{lem:characterisation-of-k-closure-and-interior} and \cref{cor:liberation-of-balls-yields-bigger-one},
            \begin{align*}
              (A^{+\theta})^{+\theta'}
              &= A^{+\theta} \rightsemiaction \ball_S(\theta')\\
              &= \parens*{\bigcup_{m \in A} \ball_S(m, \theta)} \rightsemiaction \ball_S(\theta')\\
              &= \bigcup_{m \in A} \ball_S(m, \theta) \rightsemiaction \ball_S(\theta')\\
              &= \bigcup_{m \in A} \ball_S(m, \theta + \theta')\\
              &= A^{+(\theta + \theta')}. 
            \end{align*}
      \item According to \cref{it:repeated-k-bounderies-etc:closure},
            \begin{align*}
              \boundary_{\theta'}^+ A^{+\theta}
              &= (A^{+\theta})^{+\theta'} \smallsetminus A^{+\theta}\\
              &= A^{+(\theta + \theta')} \smallsetminus A^{+\theta}.
            \end{align*}
      \item According to \cref{lem:characterisation-of-k-closure-and-interior} and \cref{it:repeated-k-bounderies-etc:closure}, 
            \begin{align*}
              A^{+(\theta - \theta')} \rightsemiaction \ball_S(\theta')
              &= (A^{+(\theta - \theta')})^{+\theta'}\\
              &= A^{+((\theta - \theta') + \theta')}\\
              &= A^{+\theta}.
            \end{align*}
            Thus, for each $m \in A^{+(\theta - \theta')}$, according to \cref{cor:liberation-of-balls-yields-bigger-one}, we have $\ball_S(m, \theta') = m \rightsemiaction \ball_S(\theta') \subseteq A^{+\theta}$ and, in particular, $m \in A^{+\theta}$. Therefore, according to \cref{lem:characterisation-of-k-closure-and-interior}, we have $A^{+(\theta - \theta')} \subseteq (A^{+\theta})^{-\theta'}$.

            According to \cref{lem:characterisation-of-k-closure-and-interior}, \cref{it:repeated-k-bounderies-etc:closure}, and \cref{def:k-interior-closure-and-boundary}, 
            \begin{align*}
              (A^{-\theta})^{+\theta'} \rightsemiaction \ball_S(\theta - \theta')
              &= ((A^{-\theta})^{+\theta'})^{+(\theta - \theta')}\\
              &= (A^{-\theta})^{+\theta' + (\theta - \theta')}\\
              &= (A^{-\theta})^{+\theta}\\
              &= A^{-\theta} \rightsemiaction \ball_S(\theta)\\
              &\subseteq A.
            \end{align*}
            Therefore, according to \cref{def:k-interior-closure-and-boundary}, we have $(A^{-\theta})^{+\theta'} \subseteq A^{-(\theta - \theta')}$. \qedhere
    \end{enumerate}
  \end{proof}

  \begin{lemma}
  \label{lem:distance-of-set-minus-closure-to-set}
    Let $k$ be a non-negative integer, and let $A$ and $A'$ be two subsets of $M$. Then, $d_S(A, A' \smallsetminus A^{+k}) \geq k + 1$.
  \end{lemma}

  \begin{proof} 
    If $A$ or $A' \smallsetminus A^{+k}$ is empty, then $d_S(A, A' \smallsetminus A^{+k}) = \infty \geq k + 1$. Otherwise, let $m' \in A' \smallsetminus A^{+k}$. According to \cite[Item~3 of Lemma~1]{wacker:garden:2016}, 
    we have $A' \smallsetminus A^{+k} = (A' \smallsetminus A)^{-k}$. Hence, according to \cref{lem:characterisation-of-k-closure-and-interior}, we have $\ball_S(m', k) \subseteq A' \smallsetminus A$. Therefore, for each $m \in A$, we have $m \notin \ball_S(m', k)$ and hence $d_S(m, m') \geq k + 1$. Thus, $d_S(A, m') \geq k + 1$. In conclusion, $d_S(A, A' \smallsetminus A^{+k}) \geq k + 1$.
  \end{proof}

  \begin{corollary}
  \label{cor:distance-of-closure-boundary-of-closure-to-set-greater-than-closure-plus-one}
    Let $k$ be a non-negative integer, let $k'$ be a positive integer, and let $A$ be a subset of $M$. Then, $d_S(A, \boundary_{k'}^+ A^{+k}) \geq k + 1$.
  \end{corollary}

  \begin{proof}
    Because $\boundary_{k'}^+ A^{+k} = (A^{+k})^{+k'} \smallsetminus A^{+k}$, this is a direct consequence of \cref{lem:distance-of-set-minus-closure-to-set}.
  \end{proof}

  \begin{lemma} 
  \label{lem:finite-set-contained-in-ball-if-gen-set-contains-neutral-element}
    Let $A$ be a finite subset of $M$ and let $S'$ be the set $\set{G_0} \cup S$. There is a non-negative integer $k \in \N_0$ such that
    \begin{equation*}
      A \subseteq \set{m \in M \suchthat \Exists \family{s_i'}_{i \in \set{1,2,\dotsc,k}} \subseteq S' \SuchThat (((m_0 \rightsemiaction s_1') \rightsemiaction s_2') \rightsemiaction \dotsb) \rightsemiaction s_k'}.
    \end{equation*}
  \end{lemma}

  \begin{proof}
    If $A$ is empty, then any $k \in \N_0$ works. Otherwise, let $k = \max_{a \in A} d_S(m_0, a)$. Because $A$ is finite, we have $k \in \N_0$. By the choice of $k$, we have $A \subseteq \ball(m_0, k)$. And, because $G_0 \in S'$ and $\blank \rightsemiaction G_0 = \identity_M$, we have $\ball(m_0, k) = \set{m \in M \suchthat \Exists \family{s_i'}_{i \in \set{1,2,\dotsc,k}} \subseteq S' \SuchThat (((m_0 \rightsemiaction s_1') \rightsemiaction s_2') \rightsemiaction \dotsb) \rightsemiaction s_k'}$. In conclusion, the stated inclusion holds.
  \end{proof}

  \section{Growth Functions And Types}
  \label{sec:growth-functions-and-types}

  In this section we recapitulate growth functions and types, more or less as presented in the monograph \enquote{Cellular Automata and Groups}\cite{ceccherini-silberstein:coornaert:2010}.

  \begin{definition} 
    Let $\gamma$ be a map from $\N_0$ to $\R_{\geq 0}$. It is called \define{growth function}\graffito{growth function} if and only if it is \define{non-decreasing}\graffito{non-decreasing}, that is to say, that 
    \begin{equation*}
      \ForEach k \in \N_0 \ForEach k' \in \N_0 \Holds \big(k \leq k' \implies \gamma(k) \leq \gamma(k')\big). \qedhere
    \end{equation*}
  \end{definition}

  \begin{definition} 
    Let $\gamma$ and $\gamma'$ be two growth functions. The growth function $\gamma$ is said to \define{dominate}\graffito{$\gamma$ dominates $\gamma'$} $\gamma'$ and we write $\gamma \dominates \gamma'$\graffito{$\gamma \dominates \gamma'$} if and only if 
    \begin{equation*}
      \Exists \alpha \in \N_+ \SuchThat \ForEach k \in \N_+ \Holds \alpha \cdot \gamma(\alpha \cdot k) \geq \gamma'(k). \qedhere
    \end{equation*}
  \end{definition} 

  \begin{definition} 
    Let $\gamma$ and $\gamma'$ be two growth functions. They are called \define{equivalent}\graffito{equivalent} and we write $\gamma \equivalent \gamma'$\graffito{$\gamma \equivalent \gamma'$} if and only if $\gamma \dominates \gamma'$ and $\gamma' \dominates \gamma$. 
  \end{definition}

  \begin{lemma}[{\cite[Proposition~6.4.3]{ceccherini-silberstein:coornaert:2010}}]\leavevmode 
    \begin{enumerate}
      \item The relation $\dominates$ is reflexive and transitive.
      \item The relation $\equivalent$ is an equivalence relation.
      \item If $\gamma_1 \equivalent \gamma_2$ and $\gamma_1' \equivalent \gamma_2'$, then $\gamma_1 \dominates \gamma_1'$ implies $\gamma_2 \dominates \gamma_2'$. \qedhere
    \end{enumerate}
  \end{lemma}

  \begin{definition} 
    Let $\gamma$ be a growth function. The equivalence class of $\gamma$ with respect to $\equivalent$ is denoted by $\equivclass{\gamma}_\equivalent$\graffito{$\equivclass{\gamma}_\equivalent$} and called \define{growth type}\graffito{growth type $\equivclass{\gamma}_\equivalent$}.
  \end{definition}


  \begin{definition} 
    Let $\Gamma$ and $\Gamma'$ be two growth types. The growth type $\Gamma$ is said to \define{dominate}\graffito{$\Gamma$ dominates $\Gamma'$} $\Gamma'$ and we write $\Gamma \dominates \Gamma'$ if and only if 
    \begin{equation*}
      \Exists \gamma \in \Gamma \Exists \gamma' \in \Gamma' \SuchThat \gamma \dominates \gamma'. \qedhere 
    \end{equation*}
  \end{definition}


  \begin{example}[{\cite[Examples~6.4.4]{ceccherini-silberstein:coornaert:2010}}]\leavevmode 
  \label{ex:growth-functions}
    \begin{enumerate}
      \item \label{it:growth-functions:id-vs-unity}
            The growth function $[k \mapsto k]$ dominates $\unityfnc$ but they are not equivalent.

            \begin{proof}
              For each $k \in \N_+$, we have $k \geq \unityfnc(k)$. But, for each $\alpha \in \N_+$, there is a $k \in \N_+$, for example $k = \alpha + 1$, such that $\alpha \unityfnc(\alpha k) = \alpha < k$.
            \end{proof}
      \item Let $r$ and $s$ be two non-negative real numbers. Then, $[k \mapsto k^r] \dominates [k \mapsto k^s]$ if and only if $r \geq s$. And, $[k \mapsto k^r] \equivalent [k \mapsto k^s]$ if and only if $r = s$. 
      \item Let $\gamma$ be a growth function such that it is a polynomial function of degree $d \in \N_0$. Then, $\gamma \equivalent [k \mapsto k^d]$. 
      \item \label{it:growth-functions:exponential-growth}
            Let $r$ and $s$ be two elements of $\R_{> 1}$. Then, $[k \mapsto r^k] \equivalent [k \mapsto s^k]$. In particular, $[k \mapsto r^k] \equivalent \exp$.

            \begin{proof}
              Without loss of generality, suppose that $r \leq s$. Then, for each $k \in \N_+$, we have $r^k \leq s^k$. Hence, $[k \mapsto r^k] \dominatedby [k \mapsto s^k]$. Moreover, let $\alpha = \ceil{\log_r s} \in \N_+$. Then, for each $k \in \N_+$,
              \begin{equation*}
                s^k = (r^{\log_r s})^k = r^{(\log_r s) k} \leq r^{\alpha k} \leq \alpha r^{\alpha k}.
              \end{equation*}
              Hence, $[k \mapsto r^k] \dominates [k \mapsto s^k]$. In conclusion, $[k \mapsto r^k] \equivalent [k \mapsto s^k]$.
            \end{proof}
      \item\label{it:growth-functions:exp-dominates-polynomials}
            Let $d$ be a non-negative integer. Then, $\exp \dominates [k \mapsto k^d]$ and $[k \mapsto k^d] \inequivalent \exp$.

            \begin{proof}
              See \cite[Examples~6.4.4 (d)]{ceccherini-silberstein:coornaert:2010}. \qedhere
            \end{proof}
    \end{enumerate}
  \end{example}

  \begin{lemma} 
  \label{lem:growth-function-dominated-by-poly-is-dominated-by-exp-and-inequivalent}
    Let $\gamma$ be a growth function and let $d$ be a non-negative integer such that $[k \mapsto k^d] \dominates \gamma$. Then, $\exp \dominates \gamma$ and $\exp \inequivalent \gamma$.
  \end{lemma}

  \begin{proof}
    According to \cref{it:growth-functions:exp-dominates-polynomials} of \cref{ex:growth-functions}, we have $\exp \dominates [k \mapsto k^d]$ and $\exp \inequivalent [k \mapsto k^d]$. Hence, because $\dominates$ is transitive and $[k \mapsto k^d] \dominates \gamma$, we have $\exp \dominates \gamma$ and $\exp \inequivalent \gamma$. 
  \end{proof}

  \section{Cell Spaces' Growth Functions and Types}
  \label{sec:cell-spaces-growth-fnc-and-types}

  In this section, let $\mathcal{R} = \ntuple{\ntuple{M, G, \leftaction}, \ntuple{m_0, \family{g_{m_0, m}}_{m \in M}}}$ be a cell space such that there is a finite and symmetric right generating set $S$ of $\mathcal{R}$.

  In \cref{def:growth-fnct-of-cell-space} we define the $S$-growth function $\gamma_S$ of $\mathcal{R}$. In \cref{lem:metric-and-generating-set} and its corollaries we show that $\gamma_S$ is dominated by $\exp$ and that the $\equivalent$-equivalence class $\equivclass{\gamma_S}_\equivalent$ does not depend on $S$. In \cref{def:growth-type-of-cell-space} we define the growth type $\gamma(\mathcal{R})$ of $\mathcal{R}$ as that equivalence class. In \cref{lem:ball-sequence-strictly-increasing-or-eventually-constant} and its corollary we relate the inclusion-behaviour of the sequence of balls to the cardinality of $M$. And in \cref{def:growth} we define the terms \enquote{exponential growth}, \enquote{sub-exponential growth}, \enquote{polynomial growth}, and \enquote{intermediate growth of $\mathcal{R}$}.

  \begin{definition} 
  \label{def:growth-fnct-of-cell-space}
    The map
    \begin{align*}
      \gamma_S \from \N_0 &\to     \N_0, \mathnote{$S$-growth function $\gamma_S$ of $\mathcal{R}$}\\
                        k &\mapsto \abs{\ball_S(k)},
    \end{align*}
    is called \define{$S$-growth function of $\mathcal{R}$}.
  \end{definition}

  \begin{remark}
  \label{rem:growth-function-of-cell-space}
    According to \cref{rem:ball-of-radius-0-contains-one-element-and-sequence-of-balls-is-monotonic}, we have $\gamma_S(0) = 1$ and the sequence $\sequence{\gamma_S(k)}_{k \in \N_0}$ is non-decreasing with respect to the partial order $\leq$. Moreover, according to \cref{rem:upper-bound-for-cardinality-of-balls}, for each non-negative integer $k \in \N_0$, we have $\gamma_S(k) \leq (1 + \abs{S})^k$.
  \end{remark}

  \begin{lemma} 
  \label{lem:metric-and-generating-set}
    Let $S'$ be a finite and symmetric right generating set of $\mathcal{R}$ and let $\alpha$ be the non-negative integer $\min\set{k \in \N_0 \suchthat \ball_S(1) \subseteq \ball_{S'}(k)}$. Then, 
    \begin{equation*}
      \ForEach m \in M \ForEach m' \in M \Holds d_{S'}(m, m') \leq \alpha \cdot d_S(m, m'),
    \end{equation*}
    in particular, 
    \begin{equation*}
      \ForEach m \in M \Holds \abs{m}_{S'} \leq \alpha \cdot \abs{m}_S. \qedhere
    \end{equation*}
  \end{lemma}

  \begin{proof}
    For each $m \in M$, let $\alpha_m = \min\set{k \in \N_0 \suchthat \ball_S(m, 1) \subseteq \ball_{S'}(m, k)}$, in particular, $\alpha_{m_0} = \alpha$.

    \proofpart{Proof of: $\ForEach m \in M \Holds \alpha_m = \alpha$} Let $m \in M$, let $k \in \N_0$, and let $g \in G_{m_0, m}$. Then, because $g \leftaction \blank$ is bijective, $\ball_S(1) \subseteq \ball_{S'}(k)$ if and only if $g \leftaction \ball_S(1) \subseteq g \leftaction \ball_{S'}(k)$. Moreover, according to \cref{lem:left-action-and-balls}, we have $g \leftaction \ball_S(1) = \ball_S(m, 1)$ and $g \leftaction \ball_{S'}(k) = \ball_{S'}(m, k)$. Therefore, $\ball_S(1) \subseteq \ball_{S'}(k)$ if and only if $\ball_S(m, 1) \subseteq \ball_{S'}(m, k)$. In conclusion, $\alpha_m = \alpha$.

    Proof by induction on the distance, that is, proof by induction on $k$ of
    \begin{multline*}
      \ForEach k \in \N_0 \ForEach m \in M \ForEach m' \in M \Holds\\
          (d_S(m, m') = k \implies d_{S'}(m, m') \leq \alpha \cdot k).
    \end{multline*}

    \proofpart{Base Case}
      Let $k = 0$. Furthermore, let $m$ and $m' \in M$ such that $d_S(m, m') = k$. Then, $m = m'$. Hence, $d_{S'}(m, m') = 0$. Therefore, $d_{S'}(m, m') \leq \alpha \cdot k$.

    \proofpart{Inductive Step}
      Let $k \in \N_0$ such that
      \begin{equation*}
        \ForEach m \in M \ForEach m' \in M \Holds (d_S(m, m') = k \implies d_{S'}(m, m') \leq \alpha \cdot k).
      \end{equation*}
      Furthermore, let $m$ and $m'' \in M$ such that $d_S(m, m'') = k + 1$. Then, there is a $\family{s_i}_{i \in \set{1, 2, \dotsc, k + 1}} \subseteq S$ such that $m' \rightsemiaction s_{k + 1} = m''$, where $m' = (((m \rightsemiaction s_1) \rightsemiaction s_2) \rightsemiaction \dotsb) \rightsemiaction s_k$. And, according to \cref{lem:truncated-minmal-path-yields-minimal-path}, we have $d_S(m, m') = k$. Therefore, according to the inductive hypothesis, $d_{S'}(m, m') \leq \alpha \cdot k$. Moreover, by definition of $\alpha_{m'}$, we have $m'' = m' \rightsemiaction s_{k + 1} \in \ball_S(m', 1) \subseteq \ball_{S'}(m', \alpha_{m'})$. Hence, because $\alpha_{m'} = \alpha$, we have $d_{S'}(m', m'') \leq \alpha_{m'} = \alpha$. In conclusion, because $d_{S'}$ is subadditive, $d_{S'}(m, m'') \leq d_{S'}(m, m') + d_{S'}(m', m'') \leq \alpha \cdot k + \alpha = \alpha \cdot (k + 1)$.
  \end{proof}

  \begin{corollary}
  \label{cor:balls-and-generating-set}
    In the situation of \cref{lem:metric-and-generating-set}, for each element $m \in M$ and each non-negative integer $k \in \N_0$, we have $\ball_S(m, k) \subseteq \ball_{S'}(m, \alpha \cdot k)$.
  \end{corollary}

  \begin{proof}
    This is a direct consequence of \cref{lem:metric-and-generating-set}, because for each element $m \in M$, each non-negative integer $k \in \N_0$, and each element $m' \in M$, if $d_S(m, m') \leq k$, then $d_{S'}(m, m') \leq \alpha \cdot k$.
  \end{proof}

  \begin{corollary}
  \label{cor:growth-functions-and-generativng-set}
    In the situation of \cref{lem:metric-and-generating-set}, for each non-negative integer $k \in \N_0$, we have $\gamma_S(k) \leq \gamma_{S'}(\alpha \cdot k)$.
  \end{corollary}

  \begin{proof}
    This is a direct consequence of \cref{cor:balls-and-generating-set}.
  \end{proof}

  \begin{definition} 
    Let $X$ be a set, and let $d$ and $d'$ be metrics on $X$. The metrics $d$ and $d'$ are called \define{Lipschitz equivalent}\graffito{Lipschitz equivalent} if and only if there are positive real numbers $\kappa$ and $\varkappa$ such that $\kappa \cdot d \leq d' \leq \varkappa \cdot d$. 
  \end{definition}

  \begin{corollary} 
    Let $S'$ be a finite and symmetric right generating set of $\mathcal{R}$. The metrics $d_S$ and $d_{S'}$ are Lipschitz equivalent.
  \end{corollary}

  \begin{proof}
    Let $\alpha = \min\set{k \in \N_0 \suchthat \ball_S(1) \subseteq \ball_{S'}(k)}$ and let $\alpha' = \min\set{k \in \N_0 \suchthat \ball_{S'}(1) \subseteq \ball_S(k)}$. If $\alpha = 0$ or $\alpha' = 0$, then $M = \set{m_0}$, hence $d_S = \nullityfnc = d_{S'}$, and therefore $d_S \leq d_{S'} \leq d_S$. Otherwise, according to \cref{lem:metric-and-generating-set}, we have $\frac{1}{\alpha} \cdot d_{S'} \leq d_S \leq \alpha' \cdot d_{S'}$. 
  \end{proof}

  \begin{corollary} 
  \label{cor:growth-functions-independent-of-generating-set-and-gamma-dominated-by-exp}
    Let $S'$ be a finite and symmetric right generating set of $\mathcal{R}$. The $S$-growth function $\gamma_S$ of $\mathcal{R}$ and the $S'$-growth function $\gamma_{S'}$ of $\mathcal{R}$ are equivalent.
  \end{corollary}

  \begin{proof}
    According to \cref{cor:growth-functions-and-generativng-set}, there is a $\alpha \in \N_0$ such that, for each $k \in \N_0$, we have $\gamma_S(k) \leq \gamma_{S'}(\alpha \cdot k)$. Hence, according to \cref{rem:growth-function-of-cell-space}, for each $k \in \N_0$, we have $\gamma_S(k) \leq (\alpha + 1) \gamma_{S'}((\alpha + 1) \cdot k)$. Therefore, $\gamma_S$ is dominated by $\gamma_{S'}$. Switching roles of $S$ and $S'$ yields that $\gamma_{S'}$ is dominated by $\gamma_S$. In conclusion, $\gamma_S$ and $\gamma_{S'}$ are equivalent.
  \end{proof}

  \begin{corollary}
  \label{cor:growth-function-dominated-by-exp}
    The $S$-growth function $\gamma_S$ of $\mathcal{R}$ is dominated by $\exp$.
  \end{corollary}

  \begin{proof}
    According to \cref{rem:growth-function-of-cell-space}, for each $k \in \N_0$, we have $\gamma_S(k) \leq r^k$, where $r = 1 + \abs{S}$. Hence, $\gamma_S \dominatedby [k \mapsto r^k]$. Moreover, according to \cref{it:growth-functions:exponential-growth} of \cref{ex:growth-functions}, we have $[k \mapsto r^k] \equivalent \exp$. In conclusion, $\gamma_S \dominatedby \exp$.
  \end{proof}

  \begin{definition}
  \label{def:growth-type-of-cell-space}
    The equivalence class $\gamma(\mathcal{R}) = \equivclass{\gamma_S}_\equivalent$ is called \define{growth type of $\mathcal{R}$}\graffito{growth type $\gamma(\mathcal{R})$ of $\mathcal{R}$}.
  \end{definition}

  \begin{lemma}[{\cite[Proposition~6.4.6]{ceccherini-silberstein:coornaert:2010}}] 
  \label{lem:growth-function-equivalent-to-unity-iff-bounded}
    Let $\gamma$ be a growth function such that $\gamma(0) > 0$. Then, $\gamma$ is equivalent to $\unityfnc$ if and only if $\gamma$ is bounded.
  \end{lemma}

  \begin{corollary} 
  \label{cor:m-finite-iff-growth-type-of-m-constant}
    The set $M$ is finite if and only if the growth types $\gamma(\mathcal{R})$ and $\equivclass{\unityfnc}_\equivalent$ are equal. 
  \end{corollary}

  \begin{proof}
    First, let $M$ be finite. Then, for each $k \in \N_0$, we have $\gamma_S(k) \leq \abs{M} = \abs{M} \cdot \unityfnc(\abs{M} \cdot k)$. 

    Secondly, let $\gamma(\mathcal{R}) = \equivclass{\unityfnc}_\equivalent$. Then, according to \cref{lem:growth-function-equivalent-to-unity-iff-bounded}, $\gamma_S$ is bounded by some $\xi \in \R_{> 0}$. Therefore, because $M = \bigcup_{k \in \N_0} \ball_S(k)$, $\sequence{\ball_S(k)}_{k \in \N_0}$ is non-decreasing with respect to $\subseteq$, and $\sequence{\gamma_S(k)}_{k \in \N_0} = \sequence{\abs{\ball_S(k)}}_{k \in \N_0}$, we have $\abs{M} \leq \sup_{k \in \N_0} \gamma_S(k) \leq \xi$. In conclusion, $M$ is finite.
  \end{proof} 

  \begin{lemma} 
  \label{lem:ball-sequence-strictly-increasing-or-eventually-constant}
    Either the sequence $\sequence{\ball_S(k)}_{k \in \N_0}$ is strictly increasing with respect to $\subseteq$ or \define{eventually constant}\graffito{eventually constant}, that is to say, that there is a non-negative integer $k \in \N_0$ such that, for each non-negative integer $k' \in \N_0$ with $k' \geq k$, we have $\ball_S(k') = \ball_S(k)$.
  \end{lemma}

  \begin{proof} 
    According to \cref{rem:ball-of-radius-0-contains-one-element-and-sequence-of-balls-is-monotonic}, the sequence $\sequence{\ball_S(k)}_{k \in \N_0}$ is non-decreasing with respect to $\subseteq$. If it is strictly increasing with respect to $\subseteq$, it is not eventually constant. Otherwise, there is a $k \in \N_0$ such that $\ball_S(k) = \ball_S(k + 1)$. We proof by induction on $k'$ that, for each $k' \in \N_0$ with $k' \geq k$, we have $\ball_S(k') = \ball_S(k)$.

    \proofpart{Base Case}
      Let $k' = k$. Then, $\ball_S(k') = \ball_S(k)$.

    \proofpart{Inductive Step} 
      Let $k' \in \N_0$ with $k' \geq k$ such that $\ball_S(k') = \ball_S(k)$. Furthermore, let $m \in \ball_S(k' + 1)$.
      \begin{description}
        \item[Case 1] $m \in \ball_S(k')$. Then, according to the inductive hypothesis, $m \in \ball_S(k)$.
        \item[Case 2] $m \notin \ball_S(k')$. Then, there is a $\family{s_i}_{i \in \set{1, 2, \dotsc, k' + 1}} \subseteq S$ such that $m' \rightsemiaction s_{k' + 1} = m$, where $m' = (((m_0 \rightsemiaction s_1) \rightsemiaction s_2) \rightsemiaction \dotsb) \rightsemiaction s_{k'}$. Hence, $m' \in \ball_S(k')$ and thus, according to the inductive hypothesis, $m' \in \ball_S(k)$. Therefore, according to \cref{lem:ball-liberation-included-in-ball-one-larger}, we have $m \in \ball_S(k + 1)$. Thus, because $\ball_S(k + 1) = \ball_S(k)$, we have $m \in \ball_S(k)$.
      \end{description}
      In either case, $m \in \ball_S(k)$. Therefore, $\ball_S(k' + 1) \subseteq \ball_S(k) \subseteq \ball_S(k') \subseteq \ball_S(k' + 1)$. In conclusion, $\ball_S(k' + 1) = \ball_S(k)$.

    In conclusion, $\sequence{\ball_S(k)}_{k \in \N_0}$ is eventually constant.
  \end{proof}

  \begin{corollary}
  \label{cor:m-infinite-iff-balls-strictly-increasing}
    The set $M$ is infinite if and only if the sequence $\sequence{\ball_S(k)}_{k \in \N_0}$ is strictly increasing with respect to $\subseteq$.
  \end{corollary}

  \begin{proof}
    First, let $M$ be infinite. Suppose that $\sequence{\ball_S(k)}_{k \in \N_0}$ is eventually constant. Then, there is a $k \in \N_0$ such that, for each $k' \in \N_0$ with $k' \geq k$, we have $\ball_S(k') = \ball_S(k)$. Hence, according to \cref{rem:ball-of-radius-0-contains-one-element-and-sequence-of-balls-is-monotonic}, we have $M = \bigcup_{k' \in \N_0, k' \geq k} \ball_S(k') = \ball_S(k)$ and therefore, according to \cref{rem:upper-bound-for-cardinality-of-balls}, the set $M$ is finite, which contradicts the precondition that $M$ is infinite. Thus, $\sequence{\ball_S(k)}_{k \in \N_0}$ is not eventually constant. In conclusion, according to \cref{lem:ball-sequence-strictly-increasing-or-eventually-constant}, the sequence $\sequence{\ball_S(k)}_{k \in \N_0}$ is strictly increasing with respect to $\subseteq$.

    Secondly, let $\sequence{\ball_S(k)}_{k \in \N_0}$ be strictly increasing with respect to $\subseteq$. Then, because $M = \bigcup_{k \in \N_0} \ball_S(k)$, the set $M$ is infinite. 
  \end{proof}

  \begin{corollary}
  \label{cor:m-infinite-iff-spheres-are-non-empty}
    The set $M$ is infinite if and only if
    \begin{equation}
    \label{eq:m-infinite-iff-spheres-are-non-empty}
      \ForEach \rho \in \N_0 \Holds \sphere_S(\rho) \neq \emptyset. \qedhere
    \end{equation}
  \end{corollary}

  \begin{proof}
    We have $\sphere_S(0) = \set{m_0} \neq \emptyset$. And, according to \cref{rem:spheres-expressed-in-terms-of-balls}, for each $\rho \in \N_+$, we have $\sphere_S(\rho) = \ball_S(\rho) \smallsetminus \ball_S(\rho - 1)$. Hence, $\sequence{\ball_S(k)}_{k \in \N_0}$ is strictly increasing with respect to $\subseteq$ if and only if \cref{eq:m-infinite-iff-spheres-are-non-empty} holds. Therefore, according to \cref{cor:m-infinite-iff-balls-strictly-increasing}, the set $M$ is infinite if and only if \cref{eq:m-infinite-iff-spheres-are-non-empty} holds.
  \end{proof}

  \begin{lemma} 
    The set $M$ is infinite if and only if the growth type of $\mathcal{R}$ dominates $\equivclass{k \mapsto k}_\equivalent$.
  \end{lemma}

  \begin{proof}
    First, let $M$ be infinite. Then, according to \cref{cor:m-infinite-iff-balls-strictly-increasing}, the sequence $\sequence{\ball_S(k)}_{k \in \N_0}$ is strictly increasing with respect to $\subseteq$. Hence, because $\ball_S(0) = \set{m_0}$, for each $k \in \N_0$, we have $\gamma_S(k) = \abs{\ball_S(k)} \geq k + 1$. In conclusion, $\gamma_S$ dominates $[k \mapsto k]$ and hence $\gamma(\mathcal{R})$ dominates $\equivclass{k \mapsto k}_\equivalent$.

    Secondly, let $M$ be finite. Then, according to \cref{cor:m-finite-iff-growth-type-of-m-constant}, we have $\mathcal{R} = \equivclass{\unityfnc}_\equivalent$. Hence, according to \cref{it:growth-functions:id-vs-unity} of \cref{ex:growth-functions}, the cell space $\mathcal{R}$ does not dominate $\equivclass{k \mapsto k}_\equivalent$.
  \end{proof}


  \begin{definition} 
  \label{def:growth}
    The cell space $\mathcal{R}$ is said to have
    \begin{enumerate} 
      \item \define{exponential growth}\graffito{exponential growth} if and only if its growth type $\gamma(\mathcal{R})$ is equal to $\equivclass{\exp}_\equivalent$;
      \item \define{sub-exponential growth}\graffito{sub-exponential growth} if and only if it does not have exponential growth;
      \item \define{polynomial growth}\graffito{polynomial growth} if and only if there is a non-negative integer $d \in \N_0$ such that $\gamma_S$ is dominated by $[k \mapsto k^d]$. 
      \item \define{intermediate growth}\graffito{intermediate growth} if and only if it has sub-exponential growth but not polynomial growth. \qedhere
    \end{enumerate}
  \end{definition}

  \begin{lemma} 
    Let $\mathcal{R}$ have polynomial growth. The cell space $\mathcal{R}$ has sub-exponential growth. 
  \end{lemma}

  \begin{proof}
    There is a $d \in \N_0$ such that $[k \mapsto k^d] \dominates \gamma_S$. Hence, according to \cref{lem:growth-function-dominated-by-poly-is-dominated-by-exp-and-inequivalent}, $\gamma_S \inequivalent \exp$. In conclusion, $\gamma(\mathcal{R}) \neq \equivclass{\exp}_\equivalent$.
  \end{proof}

  \section{Growth Rates}
  \label{sec:growth-rates}

  In this section, let $\mathcal{R} = \ntuple{\ntuple{M, G, \leftaction}, \ntuple{m_0, \family{g_{m_0, m}}_{m \in M}}}$ be a cell space such that there is a finite and symmetric right generating set $S$ of $\mathcal{R}$.

  In \cref{def:growth-rate} we define the $S$-growth rate of $\mathcal{R}$. And in \cref{lem:growth-rate-greater-than-one-iff-exp-growth} show how that growth rate and exponential growth relate to each other.

  \begin{lemma} 
  \label{lem:growth-rate-exists-and-is-greater-than-or-equal-to-one}
    The sequence $\sequence{\sqrt[k]{\gamma_S(k)}}_{k \in \N_0}$ converges to $\inf_{k \in \N_0} \sqrt[k]{\gamma_S(k)} \in \R_{\geq 1}$.
  \end{lemma}

  \begin{proof}
    According to \cref{cor:liberation-of-balls-yields-bigger-one},
    \begin{align*}
      \gamma_S(k + k') &=    \abs{\ball_S(k + k')}\\
                       &=    \abs{\ball_S(k) \rightsemiaction \ball_S(k')}\\
                       &\leq \abs{\ball_S(k)} \cdot \abs{\ball_S(k')}\\
                       &=    \gamma_S(k) \cdot \gamma_S(k').
    \end{align*}
    Hence, according to \cite[Lemma~6.5.1]{ceccherini-silberstein:coornaert:2010}, the sequence $\sequence{\sqrt[k]{\gamma_S(k)}}_{k \in \N_0}$ converges to $\inf_{k \in \N_0} \sqrt[k]{\gamma_S(k)}$. Moreover, because, for each $k \in \N_0$, we have $\gamma_S(k) \geq 1$, that limit point must be in $\R_{\geq 1}$.
  \end{proof}

  \begin{definition} 
  \label{def:growth-rate}
    The limit point $\lambda_S = \lim_{k \to \infty} \sqrt[k]{\gamma_S(k)}$ is called \define{$S$-growth rate of $\mathcal{R}$}\graffito{$S$-growth rate $\lambda_S$ of $\mathcal{R}$}.
  \end{definition}

  \begin{lemma} 
  \label{lem:growth-rate-greater-than-one-iff-exp-growth}
    The $S$-growth rate $\lambda_S$ of $\mathcal{R}$ is greater than $1$ if and only if the cell space $\mathcal{R}$ has exponential growth. 
  \end{lemma}

  \begin{proof} 
    First, let $\lambda_S > 1$. According to \cref{lem:growth-rate-exists-and-is-greater-than-or-equal-to-one}, for each $k \in \N_0$, we have $\sqrt[k]{\gamma_S(k)} \geq \lambda_S$ and hence $\gamma_S(k) \geq \lambda_S^k$. Therefore, $\gamma_S$ dominates $\lambda_S^{(\blank)}$ and, because $\lambda_S > 1$, the growth function $\lambda_S^{(\blank)}$ is equivalent to $\exp$, and thus $\gamma_S$ dominates $\exp$. Moreover, according to \cref{cor:growth-function-dominated-by-exp}, the growth function $\gamma_S$ is dominated by $\exp$. Altogether, $\gamma_S$ is equivalent to $\exp$. In conclusion, $\gamma(\mathcal{R}) = \equivclass{\gamma_S}_\equivalent = \equivclass{\exp}_\equivalent$.

    Secondly, let $\gamma(\mathcal{R}) = \equivclass{\exp}_\equivalent$. Then, $\gamma_S$ and $\exp$ are equivalent. In particular, $\gamma_S$ dominates $\exp$. Hence, there is a $\alpha \in \N_+$ such that, for each $k \in \N_+$, we have $\alpha \gamma_S(\alpha k) \geq \exp(k)$. Therefore, for each $k \in \N_+$, 
    \begin{align*}
      \sqrt[\alpha k]{\alpha} \sqrt[\alpha k]{\gamma_S(\alpha k)}
      &=    \sqrt[\alpha k]{\alpha \gamma_S(\alpha k)}\\
      &\geq \sqrt[\alpha k]{\exp(k)}\\
      &=    \sqrt[\alpha]{\euler}.
    \end{align*}
    Thus, because $\sequence{\sqrt[\alpha k]{\alpha}}_{k \in \N_+}$ converges to $1$ and $\sequence{\sqrt[\alpha k]{\gamma_S(\alpha k)}}_{k \in \N_+}$, as subsequence of $\sequence{\sqrt[k]{\gamma_S(k)}}_{k \in \N_0}$, converges to $\lambda_S$, we conclude $\lambda_S \geq \sqrt[\alpha]{\euler} > 1$.
  \end{proof}

  \begin{corollary} 
  \label{cor:growth-rate-equal-to-one-iff-subexp-growth}
    The $S$-growth rate $\lambda_S$ of $\mathcal{R}$ is equal to $1$ if and only if the cell space $\mathcal{R}$ has sub-exponential growth.
  \end{corollary}

  \begin{proof}
    This is a direct consequence of \cref{lem:growth-rate-greater-than-one-iff-exp-growth}.
  \end{proof}

  \begin{corollary} 
    Let $S'$ be a finite and symmetric right generating set of $\mathcal{R}$. The $S$-growth rate $\lambda_S$ of $\mathcal{R}$ is equal to $1$ or greater than $1$ if and only if the $S'$-growth rate $\lambda_{S'}$ of $\mathcal{R}$ is equal to $1$ or greater than $1$ respectively.
  \end{corollary}

  \begin{proof} 
    This is a direct consequence of \cref{cor:growth-rate-equal-to-one-iff-subexp-growth} and \cref{lem:growth-rate-greater-than-one-iff-exp-growth}.
  \end{proof}

  \section{Amenability, Følner Conditions/Nets, and Isoperimetric Constants} 
  \label{sec:folner-cond}

  In this section, let $\mathcal{R} = \ntuple{\ntuple{M, G, \leftaction}, \ntuple{m_0, \family{g_{m_0, m}}_{m \in M}}}$ be a finitely right generated cell space such that the stabiliser $G_0$ is finite, and let $S$ be a finite and symmetric right generating set of $\mathcal{R}$. 

  In \cref{def:isoperimetric-constant} we define the $S$-isoperimetric constant of $\mathcal{R}$, which measures, broadly speaking, the invariance under $\rightsemiaction\restriction_{M \times S}$ that a finite subset of $M$ can have, where $0$ means maximally and $1$ minimally invariant. In \cref{thm:folner-condition-for-finitely-right-generated-group-sets} we show that $\mathcal{R}$ is right amenable if and only if a kind of Følner condition holds, which in turn holds if and only if the $S$-isoperimetric constant is $0$. And in \cref{thm:k-boundary-characterisation-of-folner-net} we characterise right Følner nets using $\rho$-boundaries.

  \begin{remark}
  \label{rem:liberation-of-setminus-is-liberation-of-setminus-of-liberation-and-}
    Let $\mathfrak{g}$ and $\mathfrak{g}'$ be two elements of $G \quotient G_0$, and let $A$, $B$, and $C$ be three sets. Then,
    \begin{equation*}
      (\blank \rightsemiaction \mathfrak{g})^{-1}(A \smallsetminus B) = (\blank \rightsemiaction \mathfrak{g})^{-1}(A) \smallsetminus (\blank \rightsemiaction \mathfrak{g})^{-1}(B)
    \end{equation*}
    and
    \begin{equation*}
      \big((\blank \rightsemiaction \mathfrak{g}) \rightsemiaction \mathfrak{g}'\big)^{-1}(A) = (\blank \rightsemiaction \mathfrak{g})^{-1}\big((\blank \rightsemiaction \mathfrak{g}')^{-1}(A)\big).
    \end{equation*}
  \end{remark}

  \begin{remark} 
  \label{rem:a-setminus-b-leq-a-setminus-c-plus-c-setminus-b}
    Let $A$, $B$, and $C$ be three finite sets. Then,
    \begin{equation*}
      \abs{A \smallsetminus B} \leq \abs{A \smallsetminus C} + \abs{C \smallsetminus B}. \qedhere 
    \end{equation*}
  \end{remark}

  \begin{definition}
  \label{def:isoperimetric-constant}
    Let $E$ be a subset of $G \quotient G_0$ and let $\mathcal{F}$ be the set $\set{F \subseteq M \suchthat F \text{ is non-empty and finite}}$. The non-negative real number 
    \begin{equation*} 
      \iota_E(\mathcal{R}) = \inf_{F \in \mathcal{F}} \frac{\abs{\bigcup_{e \in E} F \smallsetminus (\blank \rightsemiaction e)^{-1}(F)}}{\abs{F}}
    \end{equation*}
    is called \define{$E$-isoperimetric constant of $\mathcal{R}$}\graffito{$E$-isoperimetric constant $\iota_E(\mathcal{R})$ of $\mathcal{R}$}.
  \end{definition}

  \begin{lemma}
  \label{lem:liberation-minus-liberation-lib-s-subseteq-bigcup-liberation-of-setminus-liberation-gzero-s}
    Let $A$ be a subset of $M$, let $\mathfrak{g}$ and $\mathfrak{g}'$ be two elements of $G \quotient G_0$, and identify $M$ with $G \quotient G_0$ by $[m \mapsto G_{m_0,m}]$. Then,
    \begin{equation*}
      (\blank \rightsemiaction \mathfrak{g})^{-1}(A) \smallsetminus \big(\blank \rightsemiaction (\mathfrak{g} \rightsemiaction \mathfrak{g}')\big)^{-1}(A)
      \subseteq \bigcup_{g_0 \in G_0} (\blank \rightsemiaction \mathfrak{g})^{-1}\big(A \smallsetminus (\blank \rightsemiaction g_0 \cdot \mathfrak{g}')^{-1}(A)\big).
    \end{equation*}
  \end{lemma}

  \begin{proof} 
    Let $m \in (\blank \rightsemiaction \mathfrak{g})^{-1}(A) \smallsetminus (\blank \rightsemiaction (\mathfrak{g} \rightsemiaction \mathfrak{g}'))^{-1}(A)$. Then, according to \cref{lem:almost-associativity-of-liberation-under-identification}, there is a $g_0 \in G_0$ such that $(m \rightsemiaction \mathfrak{g}) \rightsemiaction g_0 \cdot \mathfrak{g}' = m \rightsemiaction (\mathfrak{g} \rightsemiaction \mathfrak{g}') \notin A$. Therefore, $m \notin ((\blank \rightsemiaction \mathfrak{g}) \rightsemiaction g_0 \cdot \mathfrak{g}')^{-1}(A)$ and hence $m \in (\blank \rightsemiaction \mathfrak{g})^{-1}(A) \smallsetminus ((\blank \rightsemiaction \mathfrak{g}) \rightsemiaction g_0 \cdot \mathfrak{g}')^{-1}(A)$. Moreover, according to \cref{rem:liberation-of-setminus-is-liberation-of-setminus-of-liberation-and-}, we have $((\blank \rightsemiaction \mathfrak{g}) \rightsemiaction g_0 \cdot \mathfrak{g}')^{-1}(A) = (\blank \rightsemiaction \mathfrak{g})^{-1}((\blank \rightsemiaction g_0 \cdot \mathfrak{g}')^{-1}(A))$. Thus, according to \cref{rem:liberation-of-setminus-is-liberation-of-setminus-of-liberation-and-}, we have $(\blank \rightsemiaction \mathfrak{g})^{-1}(A) \smallsetminus ((\blank \rightsemiaction \mathfrak{g}) \rightsemiaction g_0 \cdot \mathfrak{g}')^{-1}(A) = (\blank \rightsemiaction \mathfrak{g})^{-1}(A \smallsetminus (\blank \rightsemiaction g_0 \cdot \mathfrak{g}')^{-1}(A))$. Hence, $m \in \bigcup_{g_0 \in G_0} (\blank \rightsemiaction \mathfrak{g})^{-1}(A \smallsetminus (\blank \rightsemiaction g_0 \cdot \mathfrak{g}')^{-1}(A))$. 
  \end{proof}

  \begin{theorem} 
  \label{thm:folner-condition-for-finitely-right-generated-group-sets} 
    The following statements are equivalent:
    \begin{enumerate}
      \item \label{it:folner-condition-for-finitely-right-generated-group-sets:right-amenable}
            The cell space $\mathcal{R}$ is right amenable;
      \item \label{it:folner-condition-for-finitely-right-generated-group-sets:folner-condition}
            For each positive real number $\varepsilon \in \R_{> 0}$, there is a non-empty and finite subset $F$ of $M$ such that
            \begin{equation}
            \label{eq:folner-condition-for-finitely-right-generated-group-sets:folner-condition}
              \ForEach s \in S \Holds \frac{\abs{F \smallsetminus (\blank \rightsemiaction s)^{-1}(F)}}{\abs{F}} < \varepsilon;
            \end{equation}
      \item \label{it:folner-condition-for-finitely-right-generated-group-sets:isoperimetric-constant}
            The isoperimetric constant $\iota_S(\mathcal{R})$ is $0$. \qedhere 
    \end{enumerate}
  \end{theorem}

  \begin{proof}
    \proofpart{\ref{it:folner-condition-for-finitely-right-generated-group-sets:right-amenable} $\implies$ \ref{it:folner-condition-for-finitely-right-generated-group-sets:folner-condition}}
      Let $\mathcal{R}$ be right amenable. Then, according to \cite[Main Theorem~4]{wacker:amenable:2016}, 
      there is a right Følner net in $\mathcal{R}$. Hence, according to \cite[Lemma~9]{wacker:amenable:2016}, 
      for each $\varepsilon \in \R_{> 0}$, there is a non-empty and finite $F \subseteq M$ such that \cref{eq:folner-condition-for-finitely-right-generated-group-sets:folner-condition} holds. 

    \proofpart{\ref{it:folner-condition-for-finitely-right-generated-group-sets:folner-condition} $\implies$ \ref{it:folner-condition-for-finitely-right-generated-group-sets:right-amenable}}
      For each $\varepsilon \in \R_{> 0}$, let there be a non-empty and finite $F \subseteq M$ such that \cref{eq:folner-condition-for-finitely-right-generated-group-sets:folner-condition} holds. Furthermore, let $\varepsilon' \in \R_{> 0}$, let $E \subseteq G \quotient G_0$ be finite, and identify $M$ with $G \quotient G_0$ by $[m \mapsto G_{m_0,m}]$. Then, according to \cref{lem:finite-set-contained-in-ball-if-gen-set-contains-neutral-element}, there is a $k \in \N_0$ such that
      \begin{equation*}
        E \subseteq \set{m \in M \suchthat \Exists \family{s_i'}_{i \in \set{1,2,\dotsc,k}} \subseteq S' \SuchThat (((m_0 \rightsemiaction s_1') \rightsemiaction s_2') \rightsemiaction \dotsb) \rightsemiaction s_k'},
      \end{equation*}
      where $S' = \set{G_0} \cup S$. Let $\varepsilon = \varepsilon' / (\abs{G_0}^2 \cdot k)$ and let $F \subseteq M$ be non-empty and finite such that \cref{eq:folner-condition-for-finitely-right-generated-group-sets:folner-condition} holds. Furthermore, let $e \in E$. Then, there is a $\family{s_i'}_{i \in \set{1,2,\dotsc,k}} \subseteq S'$ such that $(((m_0 \rightsemiaction s_1') \rightsemiaction s_2') \rightsemiaction \dotsb) \rightsemiaction s_k' = e$. For each $i \in \set{0,1,\dotsc,k}$, let $m_i = (((m_0 \rightsemiaction s_1') \rightsemiaction s_2') \rightsemiaction \dotsb) \rightsemiaction s_i'$ and let $F_i = (\blank \rightsemiaction m_i)^{-1}(F)$. Note that $m_k = e$ and that $F_0 = F$. Then, according to \cref{rem:a-setminus-b-leq-a-setminus-c-plus-c-setminus-b},
      \begin{align*}
        \abs{F \smallsetminus F_k} &=    \abs{F_0 \smallsetminus F_k}\\
                              &\leq \abs{F_0 \smallsetminus F_1} + \abs{F_1 \smallsetminus F_k}\\
                              &\leq \abs{F_0 \smallsetminus F_1} + \abs{F_1 \smallsetminus F_2} + \abs{F_2 \smallsetminus F_k}\\
                              &\leq \dotso\\
                              &\leq \sum_{i = 1}^k \abs{F_{i - 1} \smallsetminus F_i}.
      \end{align*}
      Let $i \in \set{1,2,\dotsc,k}$. Then, because $m_i = m_{i - 1} \rightsemiaction s_i'$, we have $F_{i - 1} \smallsetminus F_i = (\blank \rightsemiaction m_{i - 1})^{-1}(F) \smallsetminus (\blank \rightsemiaction (m_{i - 1} \rightsemiaction s_i'))^{-1}(F)$. Hence, according to \cref{lem:liberation-minus-liberation-lib-s-subseteq-bigcup-liberation-of-setminus-liberation-gzero-s}, we have $F_{i - 1} \smallsetminus F_i \subseteq \bigcup_{g_0 \in G_0} (\blank \rightsemiaction m_{i - 1})^{-1}(F \smallsetminus (\blank \rightsemiaction g_0 \cdot s_i')^{-1}(F))$. Therefore,
      \begin{align*}
        \abs{F_{i - 1} \smallsetminus F_i}
        &\leq \abs{\bigcup_{g_0 \in G_0} (\blank \rightsemiaction m_{i - 1})^{-1}(F \smallsetminus (\blank \rightsemiaction g_0 \cdot s_i')^{-1}(F))}\\
        &\leq \sum_{g_0 \in G_0} \abs{(\blank \rightsemiaction m_{i - 1})^{-1}(F \smallsetminus (\blank \rightsemiaction g_0 \cdot s_i')^{-1}(F))}.
      \end{align*}
      Thus, according to \cite[Corollary 1]{wacker:garden:2016}, 
      \begin{equation*}
        \abs{F_{i - 1} \smallsetminus F_i} \leq \sum_{g_0 \in G_0} \abs{G_0} \cdot \abs{F \smallsetminus (\blank \rightsemiaction g_0 \cdot s_i')^{-1}(F)}.
      \end{equation*}
      Hence, because $G_0 \cdot S' \subseteq S'$, $F \smallsetminus (\blank \rightsemiaction G_0)^{-1}(F) = F \smallsetminus F = \emptyset$, and \cref{eq:folner-condition-for-finitely-right-generated-group-sets:folner-condition} holds,
      \begin{align*}
        \abs{F_{i - 1} \smallsetminus F_i} &< \sum_{g_0 \in G_0} \abs{G_0} \cdot \varepsilon \cdot \abs{F}\\
                                      &= \abs{G_0}^2 \cdot \varepsilon \cdot \abs{F}\\
                                      &= \abs{G_0}^2 \cdot \frac{\varepsilon'}{\abs{G_0}^2 \cdot k} \cdot \abs{F}\\
                                      &= \frac{\varepsilon'}{k} \abs{F}.
      \end{align*}
      Therefore,
      \begin{equation*}
        \abs{F \smallsetminus F_k} \leq \sum_{i = 1}^k \abs{F_{i - 1} \smallsetminus F_i}
                              <    k \frac{\varepsilon'}{k} \abs{F}
                              =    \varepsilon' \abs{F}.
      \end{equation*}
      Thus, because $F_k = (\blank \rightsemiaction e)^{-1}(F)$,
      \begin{equation*}
        \frac{\abs{F \smallsetminus (\blank \rightsemiaction e)^{-1}(F)}}{\abs{F}} < \varepsilon'. 
      \end{equation*}
      Hence, according to \cite[Lemma~3.9]{wacker:amenable:2016}, 
      there is a right Følner net in $\mathcal{R}$. In conclusion, according to \cite[Theorem~5.1]{wacker:amenable:2016}, 
      the cell space $\mathcal{R}$ is right amenable.

    \proofpart{\ref{it:folner-condition-for-finitely-right-generated-group-sets:folner-condition} $\implies$ \ref{it:folner-condition-for-finitely-right-generated-group-sets:isoperimetric-constant}}
      Let $\varepsilon' \in \R_{> 0}$ and let $\varepsilon = \varepsilon' / \abs{S}$. Then, there is a non-empty and finite $F \subseteq M$ such that \cref{eq:folner-condition-for-finitely-right-generated-group-sets:folner-condition} holds. Therefore,
      \begin{equation*}
        \frac{\abs{\bigcup_{s \in S} F \smallsetminus (\blank \rightsemiaction s)^{-1}(F)}}{\abs{F}}
        \leq \sum_{s \in S} \frac{\abs{F \smallsetminus (\blank \rightsemiaction s)^{-1}(F)}}{\abs{F}}
        <    \abs{S} \cdot \varepsilon
        =    \varepsilon'.
      \end{equation*}
      In conclusion, $\iota_S(\mathcal{R}) = 0$.

    \proofpart{\ref{it:folner-condition-for-finitely-right-generated-group-sets:isoperimetric-constant} $\implies$ \ref{it:folner-condition-for-finitely-right-generated-group-sets:folner-condition}}
      Let $\varepsilon \in \R_{> 0}$. Then, because $\iota_S(\mathcal{R}) = 0$, there is a non-empty and finite $F \subseteq M$ such that
      \begin{equation*}
        \frac{\abs{\bigcup_{s \in S} F \smallsetminus (\blank \rightsemiaction s)^{-1}(F)}}{\abs{F}} < \varepsilon.
      \end{equation*}
      Hence, for each $s \in S$, because $F \smallsetminus (\blank \rightsemiaction s)^{-1}(F) \subseteq \bigcup_{s' \in S} F \smallsetminus (\blank \rightsemiaction s')^{-1}(F)$,
      \begin{equation*}
        \frac{\abs{F \smallsetminus (\blank \rightsemiaction s)^{-1}(F)}}{\abs{F}} < \varepsilon.
      \end{equation*}
      In conclusion, \cref{eq:folner-condition-for-finitely-right-generated-group-sets:folner-condition} holds.
  \end{proof}

  \begin{theorem} 
  \label{thm:k-boundary-characterisation-of-folner-net}
    Let $\net{F_i}_{i \in I}$ be a net in $\set{F \subseteq M \suchthat F \neq \emptyset, F \text{ finite}}$ indexed by $(I, \leq)$. It is a right Følner net in $\mathcal{R}$ if and only if 
    \begin{equation}
    \label{eq:k-boundary-characterisation-of-folner-net}
      \ForEach \rho \in \N_0 \Holds \lim_{i \in I} \frac{\abs{\boundary_\rho F_i}}{\abs{F_i}} = 0. \qedhere
    \end{equation}
  \end{theorem}

  \begin{proof}
    First, let $\net{F_i}_{i \in I}$ be a right Følner net. Furthermore, let $\rho$ be a non-negative integer. Then, for each index $i \in I$,
    we have $\boundary_\rho F_i = \boundary_{\ball(\rho)} F_i$. And, according to \cref{rem:upper-bound-for-cardinality-of-balls}, the ball $\ball(\rho)$ is finite. Hence, according to 
    \cite[Theorem~1]{wacker:garden:2016},
    \begin{equation*}
      \lim_{i \in I} \frac{\abs{\boundary_{\ball(\rho)} F_i}}{\abs{F_i}} = 0.
    \end{equation*}
    In conclusion, \cref{eq:k-boundary-characterisation-of-folner-net} holds.

    Secondly, let \cref{eq:k-boundary-characterisation-of-folner-net} hold. Furthermore, let $N$ be a finite subset of $G \quotient G_0$. Then, according to \cref{rem:ball-of-radius-0-contains-one-element-and-sequence-of-balls-is-monotonic}, there is a non-negative integer $\rho$ such that $N \subseteq \ball(\rho)$. Hence, for each index $i \in I$, according to
    \cite[Item~4 of Lemma~1]{wacker:garden:2016},
    we have $\boundary_N F_i \subseteq \boundary_{\ball(\rho)} F_i = \boundary_\rho F_i$. Therefore,
    \begin{equation*}
      \lim_{i \in I} \frac{\abs{\boundary_N F_i}}{\abs{F_i}} = 0.
    \end{equation*}
    In conclusion, according to
    \cite[Theorem~1]{wacker:garden:2016},
    the net $\net{F_i}_{i \in I}$ is a right Følner net.
  \end{proof}

  \section{Subexponential Growth and Amenability}
  \label{sec:subexp-growth-amenability}

  In this section, let $\mathcal{R} = \ntuple{\ntuple{M, G, \leftaction}, \ntuple{m_0, \family{g_{m_0, m}}_{m \in M}}}$ be a finitely right generated cell space such that the stabiliser $G_0$ is finite.

  In \cref{thm:sub-exp-implies-amenable} we show that if $\mathcal{R}$ has sub-exponential growth, then it is right amenable. And in \cref{thm:group-sub-exp-implies-cell-space-sub-exp} we show that if $G$ has sub-exponential growth, then so has $\mathcal{R}$.

  \begin{lemma}[{\cite[Lemma~6.11.1]{ceccherini-silberstein:coornaert:2010}}] 
  \label{lem:liminf-of-frac-is-liminf-of-sqrt}
    Let $\sequence{r_k}_{k \in \N_0}$ be a sequence of positive real numbers. Then,
    \begin{equation*}
      \liminf_{k \to \infty} \frac{r_{k + 1}}{r_k} \leq \liminf_{k \to \infty} \sqrt[k]{r_k}. \qedhere 
    \end{equation*}
  \end{lemma}

  \begin{main-theorem} 
  \label{thm:sub-exp-implies-amenable}
    Let the cell space $\mathcal{R}$ have sub-exponential growth. It is right amenable.
  \end{main-theorem}

  \begin{proof}
    Let $S$ be a finite and symmetric right generating set of $\mathcal{R}$. According to \cref{lem:liminf-of-frac-is-liminf-of-sqrt} and \cref{cor:growth-rate-equal-to-one-iff-subexp-growth},
    \begin{equation*}
      1 \leq \liminf_{k \to \infty} \frac{\gamma_S(k + 1)}{\gamma_S(k)}
        \leq \lim_{k \to \infty} \sqrt[k]{\gamma_S(k)}
        =    \lambda_S
        =    1.
    \end{equation*}
    Therefore, $\liminf_{k \to \infty} \frac{\gamma_S(k + 1)}{\gamma_S(k)} = 1$. 

    Let $\varepsilon \in \R_{> 0}$. Then, there is a $k \in \N_+$ such that $\frac{\gamma_S(k)}{\gamma_S(k - 1)} < 1 + \varepsilon$. Hence, $\gamma_S(k) - \gamma_S(k - 1) < \varepsilon \cdot \gamma_S(k - 1)$. 
    Furthermore, let $F = \ball_S(k)$ and let $s \in S$. Then, according to \cref{lem:ball-liberation-included-in-ball-one-larger}, we have $\ball_S(k - 1) \subseteq (\blank \rightsemiaction s)^{-1}(F)$. Therefore, because $\ball_S(k - 1) \subseteq F$ and $\gamma_S(k - 1) \leq \gamma_S(k)$,
    \begin{align*}
      \abs{F \smallsetminus (\blank \rightsemiaction s)^{-1}(F)}
      &\leq \abs{F \smallsetminus \ball_S(k - 1)}\\
      &=    \abs{F} - \abs{\ball_S(k - 1)}\\
      &=    \gamma_S(k) - \gamma_S(k - 1)\\
      &<    \varepsilon \cdot \gamma_S(k - 1)\\
      &\leq \varepsilon \cdot \gamma_S(k)\\
      &=    \varepsilon \cdot \abs{F}.
    \end{align*}
    In conclusion, according to \cref{thm:folner-condition-for-finitely-right-generated-group-sets}, the cell space $\mathcal{R}$ is right amenable.
  \end{proof}

  \begin{lemma} 
  \label{lem:group-dominates-cell-space-growth-rate}
    Let the group $G$ be finitely generated. The growth rate of $G$ dominates the one of $\mathcal{R}$.
  \end{lemma}

  \begin{proof}
    There is a finite and symmetric generating set $T$ of $G$ such that $G_0 T \subseteq T$. And, according to \cref{lem:gen-set-of-group-induces-right-gen-set-of-action}, the set $S = \set{t G_0 \suchthat t \in T} = \set{g_0 \cdot t G_0 \suchthat g_0 \in G_0, t \in T}$ is a finite and symmetric right generating set of $\mathcal{R}$.

    Let $k \in \N_0$ be a non-negative integer. Furthermore, let $m$ be an element of $\ball_S^{\mathcal{R}}(k)$. Then, there is a non-negative integer $j \in \set{0,1,2,\dotsc,k}$ and a family $\family{s_i}_{i \in \set{1,2,\dotsc,j}}$ of elements in $S$ such that
    \begin{equation*}
      m = \Big(\big((m_0 \rightsemiaction s_1) \rightsemiaction s_2\big) \rightsemiaction \dotsb\Big) \rightsemiaction s_j.
    \end{equation*}
    And, by the definition of $S$, there is a family $\family{t_i}_{i \in \set{1,2,\dotsc,j}}$ of elements in $T$ such that $\family{t_i G_0}_{i \in \set{1,2,\dotsc,j}} = \family{s_i}_{i \in \set{1,2,\dotsc,j}}$. And, because $\rightsemiaction$ is a right semi-action, there is a family $\family{g_{i,0}}_{i \in \set{1,2,\dotsc,j}}$ of elements in $G_0$ such that
    \begin{align*}
      m
      &= (((m_0 \rightsemiaction s_1) \rightsemiaction s_2) \rightsemiaction \dotsb) \rightsemiaction s_j\\
      &= m_0 \rightsemiaction t_1 g_{2,0} t_2 g_{3,0} t_3 \dotsb g_{j,0} t_j G_0\\
      &= g_{1,0} t_1 g_{2,0} t_2 g_{3,0} t_3 \dotsb g_{j,0} t_j \leftaction m_0,
    \end{align*}
    where $g_{1,0} = g_{m_0, m_0}$. And, because $G_0 T \subseteq T$, the family $\family{g_{i,0} t_i}_{i \in \set{1,2,\dotsc,j}}$ is one of elements in $T$. Hence, $m \in \ball_T^G(j) \leftaction m_0 \subseteq \ball_T^G(k) \leftaction m_0$. Therefore, $\ball_S^{\mathcal{R}}(k) \subseteq \ball_T^G(k) \leftaction m_0$ and thus
    \begin{equation*}
      \abs{\ball_S^{\mathcal{R}}(k)} \leq \abs{\ball_T^G(k) \leftaction m_0}
                                     \leq \abs{\ball_T^G(k)}.
    \end{equation*}
    Hence, $\gamma_S^{\mathcal{R}}(k) \leq \gamma_T^G(k)$. In conclusion, $\gamma_T^G$ dominates $\gamma_S^{\mathcal{R}}$ and thus $\gamma(G)$ dominates $\gamma(\mathcal{R})$.
  \end{proof}

  \begin{theorem}
  \label{thm:group-sub-exp-implies-cell-space-sub-exp}
    Let the group $G$ be finitely generated and let it have sub-exponential growth. The cell space $\mathcal{R}$ has sub-exponential growth and is right amenable.
  \end{theorem}

  \begin{proof}
    According to \cref{lem:group-dominates-cell-space-growth-rate}, the cell space $\mathcal{R}$ has sub-exponential growth. Hence, according to \cref{thm:sub-exp-implies-amenable}, it is right amenable.
  \end{proof}




\end{document}